\documentclass{tran-l}

\usepackage{amsmath,amsfonts,amssymb,amsthm,amscd,epsfig,comment,color,euscript, mathrsfs}
\usepackage[all]{xy}

\newcommand\ot{\otimes}
\newcommand\gq{/\!/} 

\newcommand\C{\mathbb{C}}

\newcommand\Z{\mathbb{Z}}

\newcommand\bv{\mathbf{v}}

\newcommand\bx{\mathbf{x}}

\newcommand\inv{^{-1}}
\newcommand\calP{\mathcal{P}}

\newcommand\NN{\Bbb N}
\newcommand\ZZ{\Bbb Z}

\newcommand\scS{\mathcal{S}}

\newcommand{\On}{\EuScript{O}}

\newcommand\g{\ensuremath{\mathfrak{g}}}
\newcommand\frg{\g}
\newcommand\lsl{\ensuremath{\mathfrak{sl}}}

\newcommand\frK{\mathfrak{K}}
\newcommand\M{\mathfrak{M}} \newcommand\frM{\mathfrak{M}}
\newcommand\fm{\mathfrak{m}}

\newcommand\frs{\mathfrak{s}}

\newcommand\frz{\mathfrak{z}}
\newcommand\bpi{{\mbox {\boldmath $\pi$}}}

\newcommand\frn{\mathfrak{n}}
\newcommand\frp{\mathfrak{p}}

\newcommand\id{\mathrm{id}}

\newcommand\rat{\mathrm{rat}}

\DeclareMathOperator{\Hom}{Hom}

\DeclareMathOperator{\End}{End}
\DeclareMathOperator{\Aut}{Aut}

\DeclareMathOperator{\ad}{ad}

\DeclareMathOperator{\Span}{Span}

\DeclareMathOperator{\Spec}{Spec}
\DeclareMathOperator{\Id}{Id}

\DeclareMathOperator{\ev}{ev}
\DeclareMathOperator{\Res}{Res}
\DeclareMathOperator{\supp}{supp}
\DeclareMathOperator{\rmU}{U}
\DeclareMathOperator{\Ker}{Ker}
\DeclareMathOperator{\rad}{rad}
\DeclareMathOperator{\Ann}{Ann}

\newtheorem{theo}{Theorem}[section]
\newtheorem{prop}[theo]{Proposition}
\newtheorem{lem}[theo]{Lemma}
\newtheorem{cor}[theo]{Corollary}

\theoremstyle{definition}
\newtheorem{defin}[theo]{Definition}
\newtheorem{example}[theo]{Example}

\newtheorem*{rem*}{Remark}
\newtheorem{rem}[theo]{Remark}

\newcommand{\lie}[1]{\mathfrak{#1}}

\numberwithin{equation}{section} \allowdisplaybreaks

\newcommand{\la}{\lambda}
\newcommand{\si}{\sigma}
\newcommand{\al}{\alpha}

\begin{document}
\title[Representations of equivariant map algebras]{Irreducible finite-dimensional representations\\ of equivariant map algebras}
\author{Erhard Neher}
\address{E.~Neher: University of Ottawa, Ottawa, Ontario, Canada}
\email{erhard.neher@uottawa.ca}
\author{Alistair Savage}
\address{A.~Savage: University of Ottawa, Ottawa, Ontario, Canada}
\email{alistair.savage@uottawa.ca}
\author{Prasad Senesi}
\address{P.~Senesi: The Catholic University of America, Washington, D.C.}
\email{senesi@cua.edu}

\date{April 12, 2010}

\begin{abstract}
Suppose a finite group acts on a scheme $X$ and a finite-dimensional
Lie algebra $\g$.  The corresponding \emph{equivariant map algebra}
is the Lie algebra $\frM$ of equivariant regular maps from $X$ to
$\g$.  We classify the irreducible finite-dimensional
representations of these algebras. In particular, we show that all
such representations are tensor products of evaluation
representations and one-dimensional representations, and we
establish conditions ensuring that they are all evaluation
representations. For example, this is always the case if $\frM$ is
perfect.

Our results can be applied to multiloop algebras, current algebras,
the Onsager algebra, and the tetrahedron algebra.  Doing so, we
easily recover the known classifications of irreducible
finite-dimensional representations of these algebras. Moreover, we
obtain previously unknown classifications of irreducible
finite-dimensional representations of other types of equivariant map
algebras, such as the generalized Onsager algebra.
\end{abstract}

\thanks{The first and second authors gratefully acknowledge support
from NSERC through their respective Discovery grants.  The third author was partially supported by the Discovery grants of the first
two authors.}

\maketitle \thispagestyle{empty}




\section*{Introduction}

When studying the category of representations of a (possibly
infinite-dimensional) Lie algebra, the irreducible
finite-dimensional representations often play an important role. Let
$X$ be a scheme and let $\g$ be a finite-dimensional Lie algebra,
both defined over an algebraically closed field $k$ of
characteristic zero and both equipped with the action of a finite
group $\Gamma$ by automorphisms. The \emph{equivariant map algebra}
$\frM=M(X,\g)^\Gamma$ is the Lie algebra of regular maps $X\to \g$
which are equivariant with respect to the action of $\Gamma$.
Denoting by $A$ the coordinate ring of $X$, an equivariant map
algebra can also be realized as the fixed point Lie algebra
$\frM=(\g \ot A)^\Gamma$ with respect to the diagonal action of
$\Gamma$ on $\g \otimes A$. The purpose of the current paper is to
classify the irreducible finite-dimensional representations of such
algebras.

One important class of examples of equivariant map algebras are the
(twisted) loop algebras which play a crucial role in the theory of
affine Lie algebras.  The description of the irreducible
finite-dimensional representations of loop algebras goes back to the
work of Chari and Pressley \cite{C}, \cite{CPunitary}, \cite{CPtw}.
Their work has had a long-lasting impact. Generalizations and more
precise descriptions of their work have appeared in many papers, for
example in Batra \cite{Bat}, Chari-Fourier-Khandai \cite{CFK},
Chari-Fourier-Senesi \cite{CFS}, Chari-Moura \cite{CM},
Feigin-Loktev \cite{FL}, Lau \cite{Lau}, Li \cite{Li}, and Rao
\cite{R,R2}.  Other examples of equivariant map algebras whose
irreducible finite-dimensional representations have been classified
are the Onsager algebra \cite{DR} and the tetrahedron algebra (or
three-point $\mathfrak{sl}_2$ loop algebra) \cite{Ha}. In all these
papers it was proven, sometimes using complicated combinatorial or
algebraic arguments and sometimes without explicitly stating so,
that all irreducible finite-dimensional representations are
evaluation representations.

In the current paper, we provide a complete classification of the
irreducible finite-dimensional representations of an arbitrary
equivariant map algebra.  This class of Lie algebras includes all
the aforementioned examples, and we obtain classification results in
these cases with greatly simplified proofs.  However, the class of
Lie algebras covered by the results of this paper is infinitely
larger than this set of examples.  To demonstrate this, we work out
some previously unknown classifications of the irreducible
finite-dimensional representations of other Lie algebras such as the
generalized Onsager algebra.

If $\frM=M(X,\g)^\Gamma$ is an equivariant map algebra, $\al \in
\frM$, and $x$ is a $k$-rational point of $X$, the image $\al(x)$ is
not an arbitrary element of $\g$, but rather an element of $\g^x =
\{ u \in \g: g\cdot u = u \mbox{ for all }g\in \Gamma_x\}$ where
$\Gamma_x= \{g \in \Gamma : g \cdot x = x\}$. For a finite subset
$\bx$ of $k$-rational points of $X$ and finite-dimensional
representations $\rho_x : \g^x \to \End_k V_x$, $x \in \bx$, we
define the associated \emph{evaluation representation} as the
composition
\[ \textstyle
    \frM \stackrel{\ev_\bx}{\longrightarrow} \bigoplus_{x \in \bx}
    \g^x \xrightarrow{\otimes_{x \in \bx} \rho_x}
    \End_k (\bigotimes_{x \in \bx} V_x),
\]
where $\ev_\bx : \alpha \mapsto (\alpha(x))_{x \in \bx}$ is
evaluation at $\bx$. Our definition is slightly more general than
the classical definition of evaluation representations, which
require the $V_x$ to be representations of $\g$ instead of $\g^x$.
We believe our definition to be more natural, and it leads to a
simplification of the classification of irreducible
finite-dimensional representations in certain cases.  For instance,
the Onsager algebra is an equivariant map algebra where $X=\Spec
k[t^{\pm 1}]$, $\g=\lsl_2(k)$, and $\Gamma=\{1, \si\}$, with $\si$
acting on $X$ by $\si\cdot x = x^{-1}$ and on $\g$ by the Chevalley
involution.  For the fixed points $x={\pm 1}\in X$, the subalgebras
$\g^x$ are one-dimensional, in fact Cartan subalgebras, while
$\g^x=\g$ for $x\ne \pm 1$.  In the classification of the
irreducible finite-dimensional representations of the Onsager
algebra given in \cite{DR}, not all irreducible finite-dimensional
representations are evaluation representations since the more
restrictive definition is used. Instead, a discussion of \emph{type}
is needed to reduce all irreducible finite-dimensional
representations (via an automorphism of the enveloping algebra) to
evaluation representations.  However, under the more general
definition of evaluation representation given in the current paper,
all irreducible finite-dimensional representations of the Onsager
algebra are evaluation representations and no discussion of type is
needed. Additionally, contrary to what has been imposed before (for
example in the multiloop case), we allow the representations
$\rho_x$ to be non-faithful. This will provide us with greater
flexibility.  Finally, we note that we use the term \emph{evaluation
representation} even when we evaluate at more than one point (i.e.
when $|\bx| > 1$). Such representations are often called tensor
products of evaluation representations in the literature, where
evaluation representations are at a single point.

By the above, the main question becomes: When does an irreducible
finite-dimensional representation of $\frM$ factor through an
evaluation map $\ev_\bx$? Surprisingly, the answer is not always. In particular, the Lie algebra may have
one-dimensional representations that are not evaluation
representations.  Any one-dimensional representation corresponds to
a linear form $\la\in \frM^*$ vanishing on $[\frM, \frM]$.  In some
cases, all such linear forms are evaluation representations, but in
other cases this is not true. Our main result
(Theorem~\ref{thm:fd-reps=eval-reps+linear-form}) is that \emph{any
irreducible finite-dimensional representation of $\frM$ is a tensor
product of an evaluation representation and a one-dimensional
representation.}

Obviously, if $\frM=[\frM, \frM]$ is perfect, then our theorem
implies that every irreducible finite-dimensional representation of
$\frM$ is an evaluation representation. For example, this is so in
the case of a multiloop algebra which is an equivariant map algebra
$\frM=M(X,\g)^\Gamma$ for $X= \Spec k[t^{\pm 1}_1, \ldots, t_n^{\pm
1}]$, $\g$ simple, and $\Gamma=\ZZ/(m_1\ZZ) \times \cdots \times
\ZZ/(m_n\ZZ)$ with the $i$-th factor of $\Gamma$ acting on the
$i$-th coordinate of $X$ by a primitive $m_i$-th root of unity. Note
that in this case $\Gamma$ acts freely on $X$ and therefore all
$\g^x=\g$. But perfectness of $\frM$ is by no means necessary for
all irreducible finite-dimensional representations to be evaluation
representations. In fact, in Section~\ref{sec:classification} we
provide an easy criterion for this to be the case, which in
particular can be applied to the Onsager algebra to see that all
irreducible finite-dimensional representations are evaluation
representations in the more general sense. On the other hand, we
also provide conditions under which not all irreducible
finite-dimensional representations of an equivariant map algebra are
evaluation representations (see
Proposition~\ref{prop:infinite-tildeX}).

An important feature of our classification is the fact that
isomorphism classes of evaluation representations are parameterized
in a natural and uniform fashion.  Specifically, for a rational
point $x \in X_\rat$, let $\mathcal{R}_x$ be the set of isomorphism
classes of irreducible finite-dimensional representations of $\g^x$
and set $\mathcal{R}_X = \bigsqcup_{x \in X_\rat} \mathcal{R}_x$.
Then there is a canonical $\Gamma$-action on $\mathcal{R}_X$, and
isomorphism classes of evaluation representations are naturally
enumerated by finitely-supported $\Gamma$-equivariant functions
$\Psi : X_\rat \to \mathcal{R}_X$ such that $\Psi(x) \in
\mathcal{R}_x$.  Thus we see that evaluation representations ``live
on orbits''.  This point of view gives a natural geometric
explanation for the somewhat technical algebraic conditions that
appear in previous classifications (such as for the multiloop
algebras).

We shortly describe the contents of the paper. After a review of
some results in the representation theory of Lie algebras in
Section~\ref{sec:review}, we introduce map algebras ($\Gamma=\{1\}$)
and equivariant map algebras (arbitrary $\Gamma$) in
Sections~\ref{sec:mapal} and~\ref{sec:equiva} and discuss old and
new examples. We introduce the formalism of evaluation
representations in Section~\ref{sec:eval-reps} and classify the
irreducible finite-dimensional representations of equivariant map
algebras in Section~\ref{sec:classification}. We show that in
general not all irreducible finite-dimensional representations are
evaluation representations, and we derive a sufficient criterion for
this to nevertheless be the case, as well as a necessary condition.
Finally, in Section~\ref{sec:applications} we apply our general
theorem to some specific cases of equivariant map algebras,
recovering previous results as well as obtaining new ones.

\subsection*{Notation} Throughout this paper, $k$ is an algebraically
closed field of characteristic zero.  For schemes, we use the
terminology of \cite{EH}.  In particular, an affine scheme $X$ is
the (prime) spectrum of a commutative associative $k$-algebra $A$.
Note that we do \emph{not} assume that $A$ is finitely generated in
general.  We say that $X$ is an affine variety if $A$ is finitely
generated and reduced, in which case we identify $X$ with the
maximal spectrum of $A$.  For an arbitrary scheme $X$, we set $A =
\mathcal{O}_X(X)$, except when the possibility of confusion exists
(for instance, when more than one scheme is being considered), in
which case we use the notation $A_X$.  We let $X_\rat$ denote the
set of $k$-rational points of $X$. Recall (\cite[p.~45]{EH}) that $x
\in X$ is a \emph{$k$-rational point of $X$} if its residue field is
$k$. The symbol $\g$ will always denote a finite-dimensional Lie
algebra (but see Remark~\ref{rem:generalization}). All tensor
products will be over $k$, unless indicated otherwise. If a group
$\Gamma$ acts on a vector space $V$ we denote by $V^\Gamma=\{v\in V
: g\cdot v = v \mbox{ for all }g\in \Gamma\}$ the subspace of fixed
points.


\section{Review of some results on representations of Lie algebras}
\label{sec:review}

For easier reference we review some mostly known results about
representations of Lie algebras. Let $L$ be a Lie algebra, not
necessarily of finite dimension. The first items concern
representations of the Lie algebra
$$L= L_1 \oplus \cdots \oplus L_n,$$
where $L_1, \ldots , L_n$ are ideals of $L$. Recall that if $V_i$,
$i=1, \ldots, n$, are $L_i$-modules, then $V_1 \otimes \dots \otimes
V_n$ is an $L$-module with action
\[
    (u_1, \dots , u_n) \cdot (v_1 \otimes  \dots \otimes  v_n)
    = \sum_{i=1}^n (v_1 \otimes
    \dots \otimes  v_{i-1} \otimes u_i \cdot v_i \otimes v_{i+1}
         \otimes \dots \otimes v_n).
\]
We will always use this $L$-module structure on $V=V_1 \otimes
\cdots \otimes V_n$. It will sometimes be useful to denote this
representation by $\rho_1 \otimes \cdots \otimes \rho_n$, where
$\rho_i : L_i \to \End_k(V_i)$ are the individual representations.
We call an irreducible  $L$-module $V$ \emph{absolutely irreducible}
if $\End_L (V) = k \Id$. By Schur's Lemma, an irreducible $L$-module
of countable dimension is absolutely irreducible.

We collect in the following proposition various well-known facts
that we will need in the current paper.

\begin{prop} \label{prop:basic-facts} \mbox{}
\begin{enumerate}
\item \label{item:tensor-prods}
Let $V_i$, $1\le i \le n$, be non-zero $L_i$-modules and let $V= V_1
\otimes \cdots \otimes V_n$.  Then $V$ is completely reducible
(respectively absolutely irreducible) if and only if all $V_i$ are
completely reducible (respectively absolutely irreducible). {\rm For
a proof, see for example \cite[\S7.4, Theorem~2]{Bo:A8} and use the
fact that the universal enveloping algebra of $L$ is $\rmU(L) =
\rmU(L_1) \otimes \cdots \otimes \rmU(L_n)$, or see
\cite[Lemma~2.7]{Li}.}

\item \label{item:n=2-case} Let $n=2$; thus $L = L_1 \oplus L_2$, and let $V$ be
an irreducible $L$-module. Suppose that $V$, considered as an
$L_1$-module, contains an absolutely irreducible $L_1$-submodule
$V_1$. Then there exists an irreducible $L_2$-module $V_2$ such that
$V\cong V_1 \otimes V_2$. The isomorphism classes of $V_1$ and $V_2$
are uniquely determined by $V$. {\rm For a proof in the case $\dim_k V <
\infty$ see, for example, \cite[\S7.7, Proposition~8]{Bo:A8} and
observe that irreducibility of $V_2$ follows from
\eqref{item:tensor-prods}. For a proof in general see
\cite[Lemma~2.7]{Li}.}

\item Suppose $L_2, \ldots, L_n$ are finite-dimensional
semisimple Lie algebras (note that there are no assumptions on
$L_1$) and let $V$ be an irreducible finite-dimensional $L$-module.
Then there exist irreducible finite-dimensional $L_i$-modules $V_i$,
$1\le i \le n$, such that $V\cong V_1 \otimes \cdots \otimes V_n$.
The isomorphism classes of the $L_i$-modules $V_i$ are uniquely
determined by $V$. {\rm This follows from \eqref{item:n=2-case} by
induction.}

\item \label{item:faithful-reps}
Let $L$ be a finite-dimensional Lie algebra over $k$. Then:

\begin{enumerate}
\item $L$ has a finite-dimensional faithful completely reducible
representation if and only if $L$ is reductive {\rm (\cite[\S6.4,
Proposition~5]{Bou:Lie1})}.

\item $L$ has a finite-dimensional faithful irreducible
representation if and only if $L$ is reductive and its center is at
most one-dimensional {\rm (\cite[\S6, Exercise~20]{Bou:Lie1})}.
\end{enumerate}

\item \label{item:one-dim-reps} We denote by $L^* = \Hom_k(L,k)$
the dual space of $L$ and identify $(L/[L,L])^*= \{ \la\in L^*:
\la([L,L])=0\}$. Any $\la \in (L/[L, L])^*$ is a one-dimensional,
hence irreducible, representation of the Lie algebra $L$, and every
one-dimensional representation of $L$ is of this form. Two linear
forms $\la, \mu \in (L/[L,L])^*$ are isomorphic as representations
if and only if $\la=\mu$.  {\rm The proof of these facts is
immediate.}
\end{enumerate}
\end{prop}

We combine some of these facts in the following lemma.

\begin{lem}\label{lem:deco} Let $L$ be a (possibly infinite-dimensional)
Lie algebra and let $\rho : L \to \End_k(V)$ be an irreducible
finite-dimensional representation. Then there exist unique
irreducible finite-dimensional representations $\rho_i : L \to
\End_k(V_i)$, $i=1,2$, such that $\rho = \rho_1 \otimes \rho_2$ and
\begin{enumerate}

\item $\rho_1 = \la$ for a unique $\la \in (L/[L,L])^*$ as in
Proposition~\ref{prop:basic-facts}\eqref{item:one-dim-reps} (hence
$V=V_2$ as $k$-vector spaces) and

\item $L/\Ker \rho_2$ is finite-dimensional semisimple.
\end{enumerate}
\end{lem}

\begin{proof} The representation $\rho$ factors  as
\[ \xymatrix@!C{
  L \ar@{->>}[dr]_\pi \ar[rr]^{\rho} & & \End_k(V)\\
    & \bar L \ar@{^(->}[ur]_{\overline{\rho}}
}\] where $\bar L = L /\Ker \rho$, $\pi$ is the canonical
epimorphism and $\bar \rho$ is a faithful representation of $\bar
L$, whence $\bar L$ is finite-dimensional. By
Proposition~\ref{prop:basic-facts}\eqref{item:faithful-reps}, $\bar
L$ is reductive. If $\bar L$ is semisimple, let $\la = 0$ and
$\rho=\rho_2$. Otherwise, by
Proposition~\ref{prop:basic-facts}\eqref{item:faithful-reps} again,
$\bar L$ has a one-dimensional center $\bar L_\frz$. Let $\bar
L_\frs = [\bar L, \, \bar L]$ be the semisimple part of $\bar L$.
Since $[\bar \rho(\bar L_\frz), \, \bar \rho (\bar L_\frs)] = 0$, it
follows that ${\bar \rho}|_{{\bar L}_\frs}$ is irreducible. By
Burnside's Theorem (see for example \cite[Chapter 4.3]{BAII}) the
associative subalgebra of $\End_k (V)$ generated by $\bar \rho (\bar
L_\frs)$ is therefore equal to $\End_k(V)$. Consequently, $\bar \rho
(\bar L_\frz) \subseteq k \Id_V$. Because $\bar \rho$ is faithful,
we actually have $\bar \rho(\bar L_\frz) = k \Id_V$. We identify
$k\Id_V \equiv k$ and then define $\rho_1 (x) = \bar \rho(\bar
x_\frz)$, where $\bar x_\frz$ is the $\bar L_\frz$-component of $\bar
x = \pi(x) \in \bar L$. By
Proposition~\ref{prop:basic-facts}\eqref{item:one-dim-reps}, $\rho_1
= \la$ for a unique $0 \ne \la\in (L/[L,L])^*$. Finally, we define
$\rho_2 : L \to \End_k(V)$ by $\rho_2(x) = \bar \rho(\bar x_\frs)$
where $\bar x_\frs$ is the $\bar L_\frs$-component of $\bar x$.
Since $\bar \rho |_{{\bar L}_\frs}$ is faithful, it follows that
$\Ker \rho_2 = \pi^{-1}(\bar L_\frz)$, whence $L/\Ker \rho_2 \cong
\bar L_\frs$ is semisimple.  This proves existence of the
decomposition. Uniqueness follows from the construction above.
\end{proof}


\section{Map algebras} \label{sec:mapal}

In this and the next section we define our main object of study --
the Lie algebra of (equivariant) maps from a scheme to another Lie
algebra -- and discuss several examples. We remind the reader that
$X$ is a scheme defined over $k$ and $\g$ is a finite-dimensional
Lie algebra over $k$.  Then $\g$ is naturally equipped with the
structure of an affine algebraic scheme, namely the affine
$n$-space, where $n=\dim \g$.  Addition and multiplication on $\g$
give rise to morphisms of schemes $\g \times_k \g \to \g$, and
multiplication by a fixed scalar yields a morphism of schemes $\g
\to \g$.

\begin{defin}[Map algebras] \label{Defmap}
We denote by $M(X,\g)$ the Lie algebra of regular functions on $X$
with values in $\g$ (equivalently, the set of morphisms of schemes
$X \to \g$), called the \emph{untwisted map algebra} or the
\emph{Lie algebra of currents} (\cite{FL}).  The multiplication in
$M(X,\g)$ is defined pointwise.  That is, for $\alpha, \beta \in
M(X,\g)$, we define $[\alpha, \beta] \in M(X,\g)$ to be the
composition
\[
  X \xrightarrow{(\alpha,\beta)} \g \times_k \g \xrightarrow{[ \cdot,
  \cdot ]} \g.
\]
The addition and scalar multiplication are defined similarly.
\end{defin}

\begin{lem} \label{lem:M-isom-O}
There is an isomorphism
$$M(X,\g) \cong \g \otimes A$$ of Lie algebras
over $A$ and hence also over $k$. The product on $\g \otimes A$ is
given by $\left[ u \otimes f, v \otimes g \right] = \left[ u,v
\right] \otimes fg$ for $u,v \in \g$ and $f,g \in A$.
\end{lem}

Because of Lemma~\ref{lem:M-isom-O}, whose proof is routine, we will
sometimes identify $M(X,\g)$ and $\g \otimes A$ in what follows.

\begin{rem}\label{categ} If $\phi : X \to Y$ is a morphism of
schemes, then $\phi^* : M(Y,\g) \to M(X,\g)$ given by $\phi^*(\alpha)
= \alpha \circ \phi$ is a Lie algebra homomorphism. For example, if
$\iota: X \hookrightarrow Y$ is an inclusion of schemes, then
$\iota^*: M(Y, \g) \rightarrow M(X, \g)$ is the restriction
$\phi^*:\alpha \mapsto \alpha|_{X}$. The assignments $X \mapsto
M(X,\g)$, $\phi \mapsto \phi^*$ are easily seen to define a
contravariant functor from the category of schemes to the category
of Lie algebras. Analogously, for fixed $X$, the assignment $\g
\mapsto M(X,\g)$ is a covariant functor from the category of
finite-dimensional Lie algebras to the category of Lie algebras.
\end{rem}

\begin{example}[Current algebras]\label{eg:current}
Let $X$ be the $n$-dimensional affine space.  Then $A
 \cong k[t_1,\dots,t_n]$ is a polynomial algebra in $n$ variables and
$M(X,\g) \cong \g \otimes k[t_1,\dots,t_n]$ is the so-called
\emph{current algebra}.
\end{example}

\begin{example}[Untwisted multiloop algebras]\label{eg:multiloop}
Let $X= \Spec A$, where $A=k[t_1^{\pm 1},\dots, t_n^{\pm 1}]$ is the
$k$-algebra of Laurent polynomials in $n$ variables. Then $M(X,\g)
\cong \g \otimes k[t_1^{\pm 1}, \dots, t_n^{\pm n}]$ is an
\emph{untwisted multiloop algebra}. In the case $n=1$, it is usually
called the \emph{untwisted loop algebra of $\g$}.
\end{example}

\begin{example}[Tetrahedron Lie algebra]
\label{eg:threepoint} If $X$ is the variety $k \setminus \{0,1\}$,
then $A \cong k[t, t^{-1}, (t-1)^{-1}]$ and $M(X,\lsl_2) \cong
\lsl_2 \otimes k[t,t^{-1},(t-1)^{-1}]$ is the \emph{three point
$\mathfrak{sl}_2$ loop algebra}. Removing any two distinct points of
$k$ results in an algebra isomorphic to $M(X,\lsl_2)$, and so there
is no loss in generality in assuming the points are 0 and 1. It was
shown in \cite{HT07} that $M(X,\lsl_2)$ is isomorphic to the
\emph{tetrahedron Lie algebra} and to a direct sum of three copies
of the \emph{Onsager algebra} (see Example~\ref{eg:onsager}). We
refer the reader to \cite{HT07} and the references cited therein for
further details.
\end{example}


\section{Equivariant map algebras} \label{sec:equiva}

Recall that we assume $\g$ is a finite-dimensional Lie algebra. We
denote the group of Lie algebra automorphisms of $\g$ by $\Aut_k
\g$.  Any Lie algebra automorphism of $\g$, being a linear map, can
also be viewed as an automorphism of $\g$ considered as a scheme. An
action of a group $\Gamma$ on $\g$ and on a scheme $X$ will always
be assumed to be by Lie algebra automorphisms of $\g$ and scheme
automorphisms of $X$.  Recall that there is an induced
$\Gamma$-action on $A$ given by
\[
    g \cdot f = f g^{-1},\quad f \in A,
       \quad g \in \Gamma,
\]
where on the right hand side we view $g^{-1}$ as the corresponding
automorphism of $X$.

\begin{defin}[Equivariant map algebras] \label{def:EMA-geometric}
Let $\Gamma$ be a group acting on a scheme $X$ and a Lie algebra
$\g$ by automorphisms. Then $\Gamma$ acts on $M(X, \g)$ by
automorphisms:
\begin{eqnarray}
  g \cdot \alpha =  g \alpha g^{-1},
    \quad \alpha \in M(X,\g), \quad g \in \Gamma, \label{eq:acmgdef}
\end{eqnarray}
where on the right hand side $g$ and $g^{-1}$ are viewed as
automorphisms of $\g$ and $X$, respectively.  We define
$M(X,\g)^\Gamma$ to be the set of fixed points under this action.
That is,
\[
    M(X,\g)^\Gamma = \{\alpha \in M(X,\g)  : \alpha g = g \alpha \
    \forall\ g \in \Gamma\}
\]
is the subalgebra of $M(X,\g)$ consisting of $\Gamma$-equivariant
maps from $X$ to $\g$. We call $M(X,\g)^\Gamma$ an \emph{equivariant
map algebra}.
\end{defin}

\begin{example}[Discrete spaces]
Suppose $X$ is a discrete (hence finite) variety. Let $X'$ be a
subset of $X$ obtained by choosing one element from each
$\Gamma$-orbit of $X$. Then
\[
    M(X,\g)^\Gamma \cong \prod_{x \in X'} \g^{\Gamma_x},\quad \alpha \mapsto (\alpha(x))_{x
    \in X'},\quad \alpha \in M(X,\g)^\Gamma,
\]
is an isomorphism of Lie algebras, where $\Gamma_x = \{g \in \Gamma
: g \cdot x = x\}$ is the stabilizer subgroup of $x$.
\end{example}

\begin{lem} \label{lem:conjugate-actions}
Let $\Gamma$ be a group acting on a scheme $X$ and a Lie algebra
$\g$. Suppose $\tau_1 \in \Aut_k \g$ and $\tau_2 \in \Aut X$. Then
we can define a second action of $\Gamma$ on $\g$ and $X$ by
declaring $g \in \Gamma$ to act by $\tau_1 g \tau_1^{-1}$ on $\g$ and
by $\tau_2 g \tau_2^{-1}$ on $X$.  Let $\frM$ and $\frM'$ be the
equivariant map algebras with respect to these two actions. That is,
\begin{align*}
  \frM &= \{\alpha \in M(X,\g) : \alpha g = g \alpha\ \forall\ g \in \Gamma \}, \\
  \frM' &= \{\beta \in M(X,\g) : \beta \tau_2 g \tau_2^{-1}  = \tau_1 g
  \tau_1^{-1} \beta \ \forall\ g \in \Gamma\}.
\end{align*}
Then $\frM \cong \frM'$ as Lie algebras.
\end{lem}

\begin{proof}
One easily checks that $\alpha \mapsto \tau_1 \circ \alpha \circ
\tau_2^{-1}$ intertwines the two $\Gamma$-actions and thus yields
the desired automorphism.
\end{proof}

The group $\Gamma$ acts naturally on $\g \otimes A$ by extending the
map $g \cdot (u \otimes f) = (g \cdot u) \otimes (g \cdot f)$ by
linearity. Define
\[
    (\g \otimes A)^\Gamma := \{\alpha \in \g
    \otimes A : g \cdot \alpha = \alpha \ \forall\ g \in \Gamma\}
\]
to be the subalgebra of $\g \otimes A$ consisting of elements fixed
by $\Gamma$.  The proof of the following lemma is immediate.

\begin{lem} \label{lem:invM-iso-O}
Let $\Gamma$ be a group acting on a scheme $X$ and a Lie algebra
$\g$. Then the isomorphism $M(X,\g) \cong \g \otimes A$ of
Lemma~\ref{lem:M-isom-O} is $\Gamma$-equivariant. In particular,
under this isomorphism $(\g \otimes A)^\Gamma$ corresponds to
$M(X,\g)^\Gamma$.
\end{lem}

\begin{rem} \label{remark on generalizations}
Let $V = \Spec A$, an affine scheme but not necessarily an affine
variety.  By assumption, $\Gamma$ acts on $X$, hence on $A$.  Thus
$\Gamma$ acts on $V$ by \cite[I-40]{EH}.  Since $A_V = A_X = A$, we
have $M(X,\g)^\Gamma \cong (\g \otimes A)^\Gamma \cong
M(V,\g)^\Gamma$. Therefore, we lose no generality in assuming that
$X$ is an affine scheme, and we will often do so in the sequel.
\end{rem}

\begin{lem}
If $\Gamma$ acts trivially on $\g$, then $M(X,\g)^\Gamma \cong
M(\Spec (A^\Gamma),\g)$, and hence $M(X,\g)^\Gamma$ is isomorphic to
an (untwisted) map algebra.
\end{lem}

\begin{proof}
The proof is straightforward.
\end{proof}

Let $\Gamma$ be a finite group. Recall that any $\Gamma$-module $B$
decomposes uniquely as a direct sum $B= B^\Gamma \oplus B_\Gamma$ of
the two $\Gamma$-submodules $B^\Gamma$ and $B_\Gamma = \Span\{
g\cdot m - m : g\in \Gamma, m\in B\}$. We let $\sharp : B \to
B^\Gamma$ be the canonical $\Gamma$-module epimorphism, given by
$m\mapsto m^\sharp = \frac{1}{|\Gamma|} \sum_{g\in \Gamma}\, g\cdot
m$. If $B$ is a (possibly nonassociative) algebra and $\Gamma$ acts
by automorphisms, then
\begin{equation}\label{eg:perfrm1}
  [B^\Gamma,\,  B^\Gamma] \subseteq B^\Gamma, \quad [B^\Gamma, B_\Gamma] \subseteq B_\Gamma,
  \quad [m_0, m^\sharp] = [m_0, m]^\sharp,
\end{equation}
for $m_0 \in B^\Gamma$ and $m\in B$.  Since $\g \ot A = (\g^\Gamma
\ot A^\Gamma) \oplus (\g^\Gamma \ot A_\Gamma) \oplus (\g_\Gamma \ot
A^\Gamma) \oplus (\g_\Gamma \ot A_\Gamma)$, we get
\[
  \frM = (\g^\Gamma \ot A^\Gamma) \oplus (\g_\Gamma \ot
  A_\Gamma)^\Gamma
\]
and
\[
  [\g^\Gamma \otimes A^\Gamma, \g^\Gamma \otimes A^\Gamma] \subseteq
  \g^\Gamma \otimes A^\Gamma,\quad [\g^\Gamma \otimes A^\Gamma,
  (\g_\Gamma \otimes A_\Gamma)^\Gamma] \subseteq (\g_\Gamma \otimes
  A_\Gamma)^\Gamma.
\]

\begin{lem} \label{lem:perfrm} We suppose that $\Gamma$ is a
finite group and use the notation above.  Then we have:
\begin{enumerate}
 \item \label{item:perfect-cond1} If $[\g^\Gamma, \g_\Gamma] = \g_\Gamma$, then
\[
  [\g^\Gamma, \g^\Gamma] \ot A^\Gamma \subseteq [\frM, \frM]
  = \big( [\frM, \frM] \cap (\g^\Gamma \ot A^\Gamma)\big) \oplus (\g_\Gamma \ot A_\Gamma)^\Gamma.
\]
  In particular, if $[\g^\Gamma, \g]=\g$, then $\frM$ is perfect.
  \item  If $\g$ is simple, $\g^\Gamma\ne \{0\}$ is perfect and the representation
  of $\g^\Gamma$ on $\g_\Gamma$ does not have a trivial non-zero subrepresentation, then
  $\g=[\g^\Gamma,\g]$, and  hence $\frM$ is perfect.
\end{enumerate}
\end{lem}

\begin{proof}\mbox{}
\begin{enumerate}
\item The inclusion $ [\g^\Gamma, \g^\Gamma] \ot A^\Gamma \subseteq [\frM, \frM]$
  is obvious since $\g^\Gamma \otimes 1 \subseteq \frM$. For $u_0\in \g^\Gamma$
  and $v\in \g_\Gamma$ we have, using
  (\ref{eg:perfrm1}), $([u_0, v] \ot f)^\sharp = [u_0 \ot 1, (v\ot
  f)^\sharp] \in [\frM, \frM]$, whence $(\g_\Gamma \ot A_\Gamma)^\Gamma = (\g_\Gamma \ot
  A_\Gamma)^\sharp \subseteq [\frM, \frM]$ by linearity of $\sharp$. The
  same argument also shows that $[\g^\Gamma, \g]=\g$ implies that $\frM$ is
  perfect.

\item
Since $\g^\Gamma$ is reductive by \cite[VII, \S1.5,
Proposition~14]{Bou:Lie78} and perfect by assumption, $\g^\Gamma$ is
semisimple.  Thus the representation of $\g^\Gamma$ on $\g_\Gamma$
is a direct sum of irreducible representations. If $U \subseteq
\g_\Gamma$ is an irreducible $\g^\Gamma$-submodule then $[\g^\Gamma,
U]\subseteq U$ is a submodule, which is non-zero by assumption on
$(\ad \g^\Gamma)|_{\g_\Gamma}$. Hence $[\g^\Gamma, U]=U$ and thus
$[\g^\Gamma, \g_\Gamma]= \g_\Gamma$.  Then $\g = [\g^\Gamma,
\g^\Gamma] \oplus [\g^\Gamma, \g_\Gamma]= [\g^\Gamma, \g]$ follows,
so that we can apply \eqref{item:perfect-cond1}.
\end{enumerate}
\end{proof}

\begin{example}[$\Gamma$ abelian]  \label{eq:abeli} Since the groups in several
of the examples discussed later will be abelian, it is convenient to
discuss the case of an arbitrary abelian  group. We assume that the
action of $\Gamma$ on $\g$ and on $A$ is diagonalizable -- a
condition which is always fulfilled if $\Gamma$ is finite. Hence,
denoting by $\Xi=\Xi(\Gamma)$ the character group of $\Gamma$, the
action of $\Gamma$ on $\g$ induces a $\Xi$-grading of $\g$, i.e.,
\[
   \textstyle \g = \bigoplus_{\xi \in \Xi} \g_\xi, \quad
    [\g_\xi, \g_\zeta] \subseteq \g_{\xi + \zeta},
\]
where $\g_\xi = \{ u \in \g :g\cdot u = \xi(g) u \mbox{ for all }
g\in \Gamma\}$ and $\xi,\zeta \in \Xi$. Thus $\g_0 = \g^\Gamma$ and
$\bigoplus_{0\ne \xi} \g_\xi = \g_\Gamma$ in the notation above.
We have a similar decomposition for $A$. The fixed point subalgebra
of the diagonal action of $\Gamma$ on $\g \ot A$ is therefore
\[
   (\g \ot A)^\Gamma = \textstyle \bigoplus_{\xi \in \Xi}
                 \g_\xi \ot A_{-\xi},
\]
where $-\xi$ corresponds to the representation dual to $\xi$. The
assumption $[\g^\Gamma, \g_\Gamma]=\g^\Gamma$ means $[\g_0,
\g_\xi]=\g_\xi$ for all $0\ne \xi$, and if this is fulfilled we have
\[
  [\frM, \frM] = \left(\sum_\xi [\g_\xi, \g_{-\xi}] \ot A_\xi
  A_{-\xi}\right) \oplus \bigoplus_{0\ne \xi} \g_\xi \otimes
  A_{-\xi}
\]
As a special case, let $\Gamma=\{1,\si\}$ be a group of order $2$.
Then $\g= \g_0 \oplus \g_1$ is a $\ZZ_2$-grading where $\g_s = \{ u
\in \g: \si \cdot u= (-1)^s u \}$ for $s= 0,1\in \ZZ_2$. Thus $\frM
= (\g \ot A)^\Gamma = (\g_0 \ot A_0)  \, \oplus \, (\g_1 \ot A_1)$,
where $A = A_0 \oplus A_1$ is the $\ZZ_2$-grading of $A$ induced by
$\si$. Hence
\[
  [\frM, \, \frM]= ([\g_0,\g_0] \ot A_0 + [\g_1, \g_1] \ot A_1^2)
  \, \oplus \,([\g_0, \g_1]\ot A_1).
\]
In particular, if $\g$ is simple and $\si\ne \Id$, then
\[
  \g = [\g,\g] = ([\g_0,\g_0] +[\g_1,\g_1])\oplus([\g_0,\g_1]).
\]
This implies $[\g_0,\g_1]=\g_1$. Also, since $\g_1 + [\g_1,\g_1]$ is
the ideal generated by $\g_1$, which must be all of $\g$, we have
$[\g_1,\g_1]=\g_0$. Thus
\begin{equation} \label{eq:derivedM}
  [\frM, \, \frM]= ([\g_0,\g_0] \ot A_0 + \g_0 \ot A_1^2)
  \, \oplus \,(\g_1\ot A_1) \quad\mbox{($|\Gamma|=2$, $\g$ simple).}
\end{equation}
Therefore, in this case $\frM$ is perfect as soon as $\g_0$ is
perfect, i.e., semisimple, or $A_0 = A_1^2$, i.e., $A=A_0 \oplus
A_1$ is a strong $\ZZ_2$-grading.
\end{example}

\begin{example}[Generalized Onsager algebra] \label{eg:onsager}
Let $X=k^\times = \Spec k[t,t^{-1}]$, $\g$ be a simple Lie algebra,
and $\Gamma=\{1,\sigma\}$ be a group of order $2$. We choose a set
of Chevalley generators $\{e_i, f_i, h_i\}$ for $\g$ and let
$\Gamma$ act on $\g$ by the standard Chevalley involution, i.e.,
\[
  \sigma(e_i) = -f_i,\ \sigma(f_i) = -e_i,\ \sigma(h_i) = -h_i.
\]
Let $\Gamma$ act on $ k[t,t^{-1}]$ by $\si \cdot t = t^{-1}$,
inducing an action of $\Gamma$ on $X$. We define the
\emph{generalized Onsager algebra $\On(\g)$} to be the equivariant
map algebra associated to these data:
\[
   \On(\g) := M(k^\times,\g)^\Gamma \cong (\g \otimes k[t,t^{-1}])^\Gamma.
\]
These algebras have been considered by G.~Benkart and M.~Lau. The
action of $\sigma$ on $\g$ interchanges the positive and negative
root spaces, and thus the dimension of the fixed point subalgebra
$\g^\Gamma$ is equal to the number of positive roots.  This fact,
together with the classification of automorphisms of order two (see,
for example, \cite[Chapter~X, \S5, Tables~II and~III]{Helga01})
determines $\g_0=\g^\Gamma$ as follows:
\medskip
\begin{center}
  \begin{tabular}{|c|c|}
  \hline
  Type of $\g$ & (Type of) $\g_0$ \\
  \hline
  $A_n$ & $\mathfrak{so}_{n+1}$ \\
  $B_n$ ($n \ge 2$) & $\mathfrak{so}_n \oplus \mathfrak{so}_{n+1}$ \\
  $C_n$ ($n \ge 2$) & $\mathfrak{gl}_n = k \oplus \lsl_n$ \\
  $D_n$ ($n \ge 4$) & $\mathfrak{so}_n \oplus \mathfrak{so}_n$ \\
  $E_6$ & $C_4$ \\
  $E_7$ & $A_7$ \\
  $E_8$ & $D_8$ \\
  $F_4$ & $C_3 \oplus A_1$ \\
  $G_2$ & $A_1 \oplus A_1$ \\
  \hline
  \end{tabular}
\end{center}
\medskip
Since $\mathfrak{so}_2$ is one-dimensional, we see that $\g^\Gamma$
is semisimple in all cases, except $\g = \lsl_2$ (type $A_1$),
$\g=\mathfrak{so}_5$ (type $B_2$) and $\g=\mathfrak{sp}_n$ (type
$C_n$). Hence, using (\ref{eq:derivedM}), we have
\begin{equation}\label{eq:derived-onsager}
  [\On(\g),\, \On(\g)] =
  \begin{cases}
    (\g_0 \ot A_1^2) \oplus  (\g_1 \ot A_1) & \mbox{if $\g=\lsl_2$,} \\
    (\mathfrak{so}_3 \otimes A_0 + \g_0 \otimes A_1^2) \oplus
    (\g_1 \otimes A_1) & \mbox{if $\g=\mathfrak{so}_5$,} \\
    (\lsl_n \otimes A_0 + \g_0 \otimes A_1^2) \oplus (\g_1
    \otimes A_1) & \mbox{if $\g=\mathfrak{sp}_n$, and} \\
    \On(\g) & \mbox{otherwise.}
  \end{cases}
\end{equation}
Therefore
\begin{equation}\label{eq:quotient-onsager}
  \On(\g)/[\On(\g),\, \On(\g)] \cong
  \begin{cases}
    A_0/A_1^2 & \mbox{if $\g=\lsl_2$,
      $\mathfrak{so}_5$, or $\mathfrak{sp}_n$, and} \\
    0 & \mbox{otherwise.}
  \end{cases}
\end{equation}
\end{example}

\begin{example} \label{eg:general-involution}
One can consider the following situation, which is even more general than
Example~\ref{eg:onsager}.  Namely, we consider the same setup except
that we allow $\sigma$ to act by an \emph{arbitrary} involution of
$\g$.    Again, by the classification of automorphisms of order two,
it is known that $\g_0$ is either semisimple or has a
one-dimensional center.  Thus we have
\begin{equation}\label{eq:quotient-general-involution}
  \frM/[\frM,\, \frM] \cong
  \begin{cases}
    0 & \mbox{if $\g_0$ is semisimple, and} \\
    A_0/A_1^2 & \mbox{otherwise.}
  \end{cases}
\end{equation}
\end{example}

\begin{rem}
For $k=\C$ it was shown in \cite{Roan91} that $\On(\lsl_2)$ is
isomorphic to the usual \emph{Onsager algebra}. This algebra was a
key ingredient in Onsager's original solution of the 2D Ising model.
The algebra $\On(\lsl_n)$, $k=\C$, was introduced in \cite{UI96},
although the definition given there differs slightly from the one
given in Example~\ref{eg:onsager}. For $\g=\lsl_n$, the Chevalley
involution of Example~\ref{eg:onsager} is given by $\sigma \cdot u =
-u^t$, while the involution used in \cite{UI96} is given by $E_{ij}
\mapsto (-1)^{i+j+1}E_{ji}$, where $E_{ij}$ is the standard
elementary matrix with the $(i,j)$ entry equal to one and all other
entries equal to zero.  This involution is equal to $\tau_1 \sigma
\tau_1^{-1}$, where $\tau_1 (u) = DuD^{-1}$ for $D= \mathrm{diag}
(\sqrt{-1},+1,\sqrt{-1},\dots)$. Furthermore, the involution of $X$
considered in \cite{UI96} is $x \mapsto (-1)^n x^{-1}$, which is
equal to $\sigma$ if $n$ is even and to $\tau_2 \sigma \tau_2^{-1}$,
where $\tau_2 \cdot x = \sqrt{-1}x$, if $n$ is odd. Therefore it
follows from Lemma~\ref{lem:conjugate-actions} that the two versions
are isomorphic.
\end{rem}

\begin{example}[Multiloop algebras]\label{eg:twistedmultiloop}
\label{eg:twisted-multiloop} Fix positive integers $n, m_1, \dots,
m_n$.  Let
\[
    \Gamma = \langle g_1,\dots, g_n : g_i^{m_i}=1,\
    g_i g_j = g_j g_i,\ \forall\ 1 \le i,j \le n
    \rangle.
\]
Then $\Xi = \Xi(\Gamma) \cong \Z/m_1\Z \times \cdots \times \Z/m_n\Z
\cong \Gamma$. Suppose $\Gamma$ acts on a semisimple $\g$. Note that this is
equivalent to specifying commuting automorphisms $\sigma_i$,
$i=1,\dots,n$, of $\g$ such that $\sigma_i^{m_i}=\id$. For $i =
1,\dots, n$, let $\xi_i$ be a primitive $m_i$-th root of unity. As
in Example~\ref{eq:abeli} we then see that $\g$ has a $\Xi$-grading
for which the homogenous subspace of degree $\bar{\mathbf{k}}$,
$\mathbf{k}=(k_1, \ldots, k_n)\in \ZZ^n$, is given by
\[
    \g_{\bar{\mathbf{k}}} = \{u \in \g :  \sigma_i(u) = \xi_i^{k_i}u\ \forall\
    i=1,\dots,n\}.
\]
Let $X=(k^\times)^n$ and define an action of $\Gamma$ on $X$ by
\[
    g_i \cdot (z_1, \dots, z_n) = (z_1, \dots, z_{i-1},
    \xi_i z_i, z_{i+1}, \dots, z_n).
\]
Then
\begin{equation} \label{eq:twisted-multiloop-def}
    M(\g,\sigma_1,\dots,\sigma_n,m_1,\dots,m_n) := M(X,\g)^\Gamma
\end{equation}
is the \emph{multiloop algebra} of $\g$ relative to $(\sigma_1,
\dots, \sigma_n)$ and $(m_1, \ldots, m_n)$. If all $\si_i=\Id$ we
recover the untwisted multiloop algebra of
Example~\ref{eg:multiloop}. If not all $\si_i=\Id$, the algebra $M$
is therefore sometimes called the \emph{twisted loop algebra}.

With $\Gamma$ and $X$ as above, the induced action of $\Gamma$ on $A
= k[t_1^{\pm 1}, \dots, t_n^{\pm 1}]$ is given by
\[
    \sigma_i \cdot f(t_1, \dots, t_n) = f(t_1, \dots, t_{i-1},
    \xi_i^{-1} t_i, t_{i+1}, \dots, t_n),\quad f \in
    k[t_1^{\pm 1},\dots,t_n^{\pm 1}].
\]
Hence
\begin{equation} \label{eq:twisted-multiloop-diagonal}
    M(\g,\sigma_1,\dots,\sigma_n,m_1,\dots,m_n)
     \cong (\g \otimes k[t_1^{\pm 1}, \dots, t_n^{\pm 1}])^\Gamma
    \cong
     \bigoplus_{\mathbf{k} \in \Z^n} \g_{\bar {\mathbf{k}}}
           \otimes k t^{\mathbf{k}},
\end{equation}
where $t^{\mathbf{k}} = t_1^{k_1} \cdots t_n^{k_n}$.

Loop and multiloop algebras play an important role in the theory of
affine Kac-Moody algebras \cite{Kac90}, extended affine Lie
algebras, and Lie tori. The connection between the last two classes
of Lie algebras is the following: The core and the centerless core
of an extended affine Lie algebra is a Lie torus or centerless Lie
torus respectively, every Lie torus arises in this way, and there is
a precise construction of extended affine Lie algebras in terms of
centerless Lie tori \cite{Neh}. Any centerless Lie torus whose
grading root system is not of type A can be realized as a multiloop
algebra \cite{abfp2}.
\end{example}

While most classification results involving (special cases of)
equivariant map algebras in the literature use abelian groups, we
will see that the general classification developed in the current
paper only assumes that the group $\Gamma$ is finite.  Therefore,
for illustrative purposes, we include here an example of an
equivariant map algebra where the group $\Gamma$ is not abelian.  We
will see in Section~\ref{subsec:applications-nonabelian} that the
representation theory of this algebra is quite interesting.

\begin{example}[A non-abelian example] \label{eg:nonabelian}
Let $\Gamma = S_3$, the symmetric group on 3 objects, $X=
\mathbb{P}^1 \setminus \{0,1,\infty\}$ and $\g=\mathfrak{so}_8$, the
simple Lie algebra of type $D_4$.  The symmetry group of the Dynkin
diagram of type $D_4$ is isomorphic to $S_3$, and so $\Gamma$ acts
naturally on $\g$ by diagram automorphisms (see
Remark~\ref{rem:diagram-autom-case}). Now, given points $x_i, y_i$,
$i=1,2,3$, of $\mathbb{P}^1$, there is a unique M\"obius
transformation of $\mathbb{P}^1$ mapping $x_i$ to $y_i$, $i =
1,2,3$. Thus, for any permutation $\sigma$ of the points
$\{0,1,\infty\}$, there is a unique M\"obius transformation of
$\mathbb{P}^1$ which induces $\sigma$ on the set $\{0,1,\infty\}$.
Hence each permutation $\sigma$ naturally corresponds to an
automorphism of $X$, which we also denote by $\sigma$. Therefore we
can form the equivariant map algebra $\frM=M(X,\g)^\Gamma$. Now, the
subalgebra of $\g$ fixed by the unique order three subgroup of
$\Gamma$ is the simple Lie algebra of type $G_2$ (see
\cite[Proposition~8.3]{Kac90}). One easily checks that this
subalgebra is fixed by all of $\Gamma$. Therefore $\g^\Gamma$ is
perfect. By \cite[Proposition~8.3]{Kac90}, the representation of
$\g^\Gamma$ on $\g_\Gamma$ is a direct sum of two $7$-dimensional
irreducible representations.  Thus, by Lemma~\ref{lem:perfrm},
$\frM$ is perfect.
\end{example}


\section{Evaluation representations} \label{sec:eval-reps}

From now on we assume that $\Gamma$ is a \emph{finite} group, acting
on an affine scheme $X$ and a (finite-dimensional) Lie algebra $\g$
by automorphisms.  We abbreviate $\frM=M(X,\g)^\Gamma$.

\begin{defin}[Restriction]
Let $Y$ be a subscheme of $X$.  Then, as in Remark~\ref{categ}, we
have the \emph{restriction} Lie algebra homomorphism
\[
    \Res^X_Y : M(X,\g)^\Gamma \to M(Y,\g),\quad \Res^X_Y(\alpha) =
    \alpha|_Y,\quad \alpha \in M(X,\g)^\Gamma.
\]
If $Y$ is a $\Gamma$-invariant subscheme, the image of $\Res^X_Y$ is
contained in $M(Y,\g)^\Gamma$.
\end{defin}

\begin{defin}[Evaluation]
Given a finite subset $\bx \subseteq X_\rat$, we define the
corresponding \emph{evaluation map}
\[ \textstyle
    \ev_\bx : M(X,\g)^\Gamma \to \bigoplus_{x \in \bx} \g,\quad
    \alpha \mapsto (\alpha(x))_{x \in \bx},\quad \alpha \in
    M(X,\g)^\Gamma.
\]
\end{defin}

\begin{defin}
For a subset $Z$ of $X$ we define
\[
  \Gamma_Z = \{g \in \Gamma : g \cdot z = z \ \forall\ z \in Z\} \quad
  \text{and} \quad \g^Z = \g^{\Gamma_Z}.
\]
Obviously, $\Gamma_Z$ is a subgroup of $\Gamma$ and $\g^Z$ is a
subalgebra of $\g$.  In particular, for any $x \in X$, we put
$\Gamma_x = \Gamma_{\{x\}}$ and $\g^x = \g^{\{x\}}$.
\end{defin}

\begin{lem} \label{lem:restriction-image} \mbox{}
\begin{enumerate}
  \item \label{lem-item1:restriction-image}
  Let $Z$ be a subset of $X_\rat$. Then $\al(Z) \subseteq  \g^Z$
  for all $\al \in \M(X,\g)^\Gamma$. In particular, $\al(x) \in \g^x$ for
  all $x\in X_\rat$.

  \item \label{lem-item2:restriction-image}
  Let $Z$ be a $\Gamma$-invariant subscheme of $X$ for which the
  restriction $A_X \to A_Z$, $f \mapsto f|_Z$, is surjective. Then
  the restriction map $\Res^X_Z : M(X,\g)^\Gamma \to M(Z,\g)^\Gamma$ is also
  surjective.
\end{enumerate}
\end{lem}

\begin{proof}
Part \eqref{lem-item1:restriction-image}  is immediate from the
definitions. In \eqref{lem-item2:restriction-image},
finite-dimensionality of $\g$ implies that the restriction $M(X,\g)
\to M(Z,\g)$ is surjective. Since $\Gamma$ acts completely reducibly
on $M(X,\g)$ and $M(Z,\g)$, the restriction is then also surjective
on the subalgebras of $\Gamma$-invariants.
\end{proof}

\begin{defin}
We denote by $X_n$ the set of $n$-element subsets $\bx \subseteq
X_\rat$ consisting of $k$-rational points and having the property
that $y \not \in \Gamma \cdot x$ for distinct $x,y \in \bx$.
\end{defin}

\begin{cor} \label{cor:twisted-surjectivity}
For $\bx \in X_n$ the image of $\ev_\bx$ is $\bigoplus_{x \in \bx}
\g^x$.
\end{cor}

\begin{proof}
Let $Z= \bigcup_{x \in \bx} \Gamma \cdot x$. Then $Z$ is a
$\Gamma$-invariant closed subvariety. Hence $A_X \to A_Z$ is
surjective, and therefore $\Res^X_Z : M(X,\g)^\Gamma \to
M(Z,\g)^\Gamma$ is also surjective by
Lemma~\ref{lem:restriction-image}.  It is immediate that $\ev_\bx : M(Z,\g)^\Gamma
\to \bigoplus_{x \in \bx} \g^x$ is surjective.
\end{proof}

\begin{defin}[Evaluation representation] \label{def:evaluation-rep}
Fix a finite subset $\bx \subseteq X_\rat$ and let $\rho_x : \g^x
\to \End_k V_x$, $x \in \bx$, be representations of $\g^x$ on the
vector spaces $V_x$. Then define $\ev_\bx (\rho_x)_{x \in \bx} $ to
be the composition
\[ \textstyle
    M(X,\g)^\Gamma \xrightarrow{\ev_\bx} \bigoplus_{x \in \bx}
    \g^x \xrightarrow{\otimes_{x \in \bx} \rho_x}
    \End_k (\bigotimes_{x \in \bx} V_x).
\]
This defines a representation of $M(X,\g)^\Gamma$ on $\bigotimes_{x
\in \bx} V_x$ called a \emph{(twisted) evaluation representation}.
\end{defin}

\begin{rem} \label{rem:new-eval-rep}
We note some important distinctions between
Definition~\ref{def:evaluation-rep} and other uses of the term
\emph{evaluation representation} in the literature.  First of all,
some authors reserve the term \emph{evaluation representation} for
the case $n=1$ and would refer to the more general case as a tensor
product of evaluation representations.  Furthermore, traditionally
the $\rho_x$ are representations of $\g$ instead of $\g^x$ and are
required to be faithful. In the case that $\g^x = \g$ for all $x \in
X$ (for instance, if $\Gamma$ acts freely on $X$), this of course
makes no difference.  However, we will see in the sequel that the
more general definition of evaluation representation given above
allows for a more uniform classification of irreducible
finite-dimensional representations.
\end{rem}

\begin{prop} \label{prop:eval-irred}
Let $\bx \in X_n$, and for $x \in \bx$ let $\rho_x : \g^x \to \End_k
V_x$ be an irreducible finite-dimensional representation of $\g^x$.
Then the evaluation representation $\ev_\bx (\rho_x)_{x \in \bx}$ is
an irreducible finite-dimensional representation of
$M(X,\g)^\Gamma$.
\end{prop}

\begin{proof}
Since $\ev_\bx$ is surjective, this follows from
Proposition~\ref{prop:basic-facts}\eqref{item:tensor-prods}.
\end{proof}

\begin{cor}
If $\bx$ is in $X^n$ (but not necessarily in $X_n$), $\g^x$ is
semisimple for all $x \in \bx$, and $\rho_x$ is an arbitrary
finite-dimensional representation of $\g^x$ for each $x \in \bx$,
then the evaluation representation $\ev_\bx (\rho_x)_{x \in \bx}$ is
completely reducible.
\end{cor}

\begin{proof}
This follows from Proposition~\ref{prop:eval-irred} and complete
reducibility of finite-dimensional representations of each $\g^x$.
\end{proof}

By abuse of notation, we will sometimes denote a representation of
$\g$ by the underlying vector space $V$.  Then $\ev_x V$, $x \in X$,
will denote the corresponding evaluation representation of
$M(X,\g)^\Gamma$. Note that, with the notation of
Definition~\ref{def:evaluation-rep}, we have
\begin{equation} \label{eq:tensor-prod-evals}
    \ev_\bx (\rho_x)_{x \in \bx} \cong \bigotimes_{x \in \bx} \ev_x V_x.
\end{equation}

Note that $X_\rat$ is $\Gamma$-invariant and that $X_\rat=X$ if $X$ is
an affine variety.  Let $x \in X_\rat$ and $g\in \Gamma$. Since
$\Gamma_{g \cdot x} = g \Gamma_x g^{-1}$ we see that $\g^{g\cdot x}
= g \cdot \g^x$. Hence if $\rho$ is a representation of $\g^x$, then
$\rho \circ g^{-1}$ is a representation of $\g^{g\cdot x}$. Let
$\mathcal{R}_x$ denote the set of isomorphism classes of irreducible
finite-dimensional representations of $\g^x$, and put
$\mathcal{R}_X=\bigsqcup_{x \in X_\rat} \mathcal{R}_x$. Then
$\Gamma$ acts on $\mathcal{R}_X$ by
\[
  \Gamma \times \mathcal{R}_X \to \mathcal{R}_X,\quad (g,[\rho]) \mapsto g
  \cdot [\rho] := [\rho \circ g^{-1}] \in \mathcal{R}_{g \cdot x},
\]
where $[\rho] \in \mathcal{R}_x$ denotes the isomorphism class of a
representation $\rho$ of $\g^x$.

\begin{defin} \label{def:E}
Let $\mathcal{E}$ denote the set of finitely supported
$\Gamma$-equivariant functions $\Psi : X_\rat \to \mathcal{R}_X$
such that $\Psi(x) \in \mathcal{R}_x$. Here the support $\supp \Psi$
of $\Psi \in \mathcal{E}$ is the set of all $x \in X_\rat$ for which
$\Psi(x) \ne 0$, where $0$ denotes the isomorphism class of the
trivial representation.
\end{defin}

For isomorphic representations $\rho$ and $\rho'$ of $\g^x$, the
evaluation representations $\ev_x \rho$ and $\ev_x \rho'$ are
isomorphic. Therefore, for $[\rho] \in \mathcal{R}_x$, we can define
$\ev_x [\rho]$ to be the isomorphism class of $\ev_x \rho$, and this
is independent of the representative $\rho$. Similarly, for a finite
subset $\bx \subseteq X_\rat$ and representations $\rho_x$ of $\g^x$
for $x \in \bx$, we define $\ev_\bx ([\rho_x])_{x \in \bx}$ to be
the isomorphism class of $\ev_\bx (\rho_x)_{x \in \bx}$.

\begin{rem} \label{rem:diagram-autom-case}
Suppose $\g$ is a semisimple Lie algebra.  In the case that $\Gamma$
is cyclic and acts on $\g$ by admissible diagram automorphisms (no
edge joins two vertices in the same orbit), there exists a simple
description of $\mathcal{E}$ which can be seen as follows. Let $I$
be the set of vertices of the Dynkin diagram of $\g$. An action of
$\Gamma$ on this Dynkin diagram gives rise to an action of $\Gamma$
on $\g$ via
\[
  g \cdot h_i = h_{g \cdot i},\quad g \cdot e_i = e_{g \cdot
  i},\quad g \cdot f_i = f_{g \cdot i},\quad g \in \Gamma,
\]
where $\{h_i, e_i, f_i\}_{i \in I}$ is a set of Chevalley generators
of $\g$. We then have a natural action of $\Gamma$ on the weight
lattice $P$ of $\g$ given by
\[
  g \cdot \omega_i = \omega_{g \cdot i},\quad g \in \Gamma,
\]
where $\{\omega_i\}_{i \in I}$ is the set of fundamental weights of
$\g$. Now, isomorphism classes of irreducible finite-dimensional
representations of $\g$ are naturally enumerated by the set of
dominant weights $P^+$ by associating to $\lambda \in P^+$ the
isomorphism class of the irreducible highest weight representation
of highest weight $\lambda$.  Let $\tilde{\mathcal{E}}$ denote the
set of $\Gamma$-equivariant functions $X_\rat \to P^+$ with finite
support. It follows that for $\Psi \in \tilde{\mathcal{E}}$ and $x
\in X_\rat$ we have $\Psi(x) \in (P^+)^{\Gamma_x}$, where
$(P^+)^{\Gamma_x}$ denotes the set of $\Gamma_x$-invariant elements
of $P^+$.  There is a canonical bijection between $(P^+)^{\Gamma_x}$
and the positive weight lattice of $\g^x$ (see
\cite[Proposition~14.1.2]{Lus93}), and so we can associate to
$\Psi(x)$ the isomorphism class of the corresponding representation
of $\g^x$. Thus, we have a natural bijection between
$\tilde{\mathcal{E}}$ and $\mathcal{E}$. Therefore, in the case that
$\Gamma$ acts on $\g$ by admissible diagram automorphisms, the
evaluation representations are naturally enumerated by
$\Gamma$-equivariant maps from $X_\rat$ to the positive weight
lattice $P^+$ of $\g$.  In the case that $\Gamma$ acts freely on
$X$, we can drop the assumption that the diagram automorphisms be
admissible.
\end{rem}

\begin{lem} \label{lem:twisted-eval-invariance}
Suppose $\Psi \in \mathcal{E}$ and $x \in \supp \Psi$.  Then for all
$g \in \Gamma$,
\[
  \ev_x \Psi (x) = \ev_{g \cdot x} \left( g \cdot \Psi(x) \right) =
  \ev_{g \cdot x} \Psi(g \cdot x).
\]
\end{lem}

\begin{proof}
For any $g \in \Gamma$ and representation $\rho$ of $\g^x$, the
following diagram commutes:
\[
    \xymatrix{ & \g^x \ar[dd]^g \ar[dr]^\rho & \\
    M(X,\g)^\Gamma \ar[ur]^{\ev_x} \ar[dr]_{\ev_{g \cdot x}} & & \End_k V \\
    & \g^{g \cdot x} \ar[ur]_{\rho \circ g^{-1}} & }
\]
Thus, $\ev_x \rho = \ev_{g \cdot x} (\rho \circ g^{-1})$ and the
result follows.
\end{proof}

\begin{defin} \label{def:eval-rep}
For $\Psi \in \mathcal{E}$, we define $\ev_\Psi = \ev_\bx
(\Psi(x))_{x \in \bx}$, where $\bx \in X_n$ contains one element of
each $\Gamma$-orbit in $\supp \Psi$.  By
Lemma~\ref{lem:twisted-eval-invariance}, $\ev_\Psi$ is independent
of the choice of $\bx$.  If $\Psi$ is the map that is identically 0
on $X$, we define $\ev_\Psi$ to be the isomorphism class of the
trivial representation of $\frM$.  Thus $\Psi \mapsto \ev_\Psi$
defines a map $\mathcal{E} \to \scS$, where $\scS$ denotes the set
of isomorphism classes of irreducible finite-dimensional
representations of $\frM$.
\end{defin}

\begin{prop} \label{prop:inj}
The map $\mathcal{E} \to \scS$, $\Psi \mapsto \ev_\Psi$, is
injective.
\end{prop}

\begin{proof}
Suppose $\Psi \ne \Psi' \in \mathcal{E}$.  Then there exists $x \in
X_\rat$ such that $\Psi(x) \ne \Psi'(x)$.  Without loss of
generality, we may assume $\Psi(x) \ne 0$.  Let
\[
  m = (\dim \ev_\Psi)/(\dim \Psi(x)),\qquad m' = (\dim
  \ev_{\Psi'})/(\dim \Psi'(x)),
\]
where the dimension of an isomorphism class of representations is
simply the dimension of any representative of that class and
$m'=\dim \ev_{\Psi'}$ if $\Psi'(x) = 0$.  By
\eqref{eq:tensor-prod-evals}, $m$ and $m'$ are positive integers. By
Corollary~\ref{cor:twisted-surjectivity}, there exists a subalgebra
$\mathfrak{a}$ of $\frM$ such that $\ev_x (\mathfrak{a}) = \g^x$ and
$\ev_{x'} (\mathfrak{a})=0$ for all $x' \in \left( \supp \Psi \cup
\supp \Psi' \right) \setminus \{\Gamma \cdot x\}$.  Then
\[
  \ev_\Psi|_{\mathfrak{a}} = \Psi(x)^{\oplus m},\quad
  \ev_{\Psi'}|_{\mathfrak{a}}= \Psi'(x)^{\oplus m'}.
\]
Since $\Psi(x) \ne \Psi'(x)$, we have $\ev_\Psi \ne \ev_{\Psi'}$. In
the above, we have used the convention that the restriction of an
isomorphism class is the isomorphism class of the restriction of any
representative and that a direct sum of isomorphism classes is the
isomorphism class of the corresponding direct sum of
representatives.
\end{proof}

\begin{lem}\label{lem:AGammamod} Let $\frK$ be an ideal of $\frM$. If
$\frM/\frK$ does not contain non-zero solvable ideals, e.g. if
$\frM/\frK$ is finite-dimensional semisimple, then $\frK$ is an
$A^\Gamma$-submodule of $\frM$.
\end{lem}

\begin{proof}
Let $f\in A^\Gamma$. Since $A^\Gamma$ is contained in the centroid
of $\frM$, the $k$-subspace $f\frK$ is an ideal of $\frM$: $[\frM,
f\frK] = f[\frM, \frK]\subseteq f\frK$. Let $\pi : \frM \to
\frM/\frK$ be the canonical epimorphism.  Since $[f\frK + \frK,
f\frK + \frK] \subseteq [\frK + f \frK + f^2\frK, \frK] \subseteq
\frK$, the ideal $\pi(f\frK)$ of $\frM/\frK$ is abelian, whence
$f\frK \subseteq \frK$.
\end{proof}

\begin{prop} \label{prop:fixed-point-alg-eval}
  Suppose $I$ is an ideal of $\mathfrak{M}$ such that $\mathfrak{M}/I = N_1
  \oplus \dots \oplus N_s$ with $N_i$ a finite-dimensional simple
  Lie algebra for $i=1,\dots,s$.  Let $\pi : \mathfrak{M} \to
  \mathfrak{M}/I$ denote the canonical projection and for
  $i=1,\dots,s$, and let $\pi_i : \mathfrak{M} \to N_i$ denote the map
  $\pi$ followed by the projection from $\mathfrak{M}/I$ to $N_i$.
  Then there exist $x_1, \dots, x_s \in X_\rat$ such that
  \begin{equation}
    \pi(f \alpha) = (f(x_1) \pi_1(\alpha), \dots, f(x_s)
    \pi_s(\alpha)) \ \forall\ f \in A^\Gamma,\ \alpha \in \mathfrak{M}.
  \end{equation}
\end{prop}

\begin{proof}
  It suffices to show that for $i=1,\dots,s$, $f \in A^\Gamma$ and
  $\alpha \in \mathfrak{M}$ we have $\pi_i(f \alpha) = f(x_i)
  \pi_i(\alpha)$ for some $x_i \in X$.  Since $N_i$ is simple, the
  action of $\mathfrak{M}$ on $N_i$ induced by the adjoint action
  is irreducible.  By Lemma~\ref{lem:AGammamod}, $N_i$ is an
  $A^\Gamma$-module and $\pi_i$ is an $A^\Gamma$-module homomorphism.
  Since the action of $A^\Gamma$ commutes with the action of
  $\mathfrak{M}$, we have that $A^\Gamma$ must act by scalars and thus
  as a character $\chi : A^\Gamma \to k$.  This character corresponds
  to evaluation at a point $\tilde x_i \in X{\gq}\Gamma := \Spec A^\Gamma$.  Choosing any
  $x_i$ in the preimage of $\tilde x_i$ under the canonical
  projection $X \to X{\gq}\Gamma$, the result follows.
\end{proof}

\begin{rem} If $s=1$, then in the graded setting
these types of maps have been studied extensively in \cite{abfp}.
There $\rho$ is a character of the full centroid of $\frM$, while in
the above $A^\Gamma$ is a priori only a subalgebra of the centroid
of $\frM$.
\end{rem}

\begin{prop} \label{prop:lifting}
If $\supp \Psi \subseteq \{x \in X : \g^x = \g\}$, then any
evaluation representation in the isomorphism class $\ev_\Psi$ of $M(X,\g)^\Gamma$
is obtained by restriction from an evaluation representation of the
untwisted map algebra $M(X,\g)$.  In particular, the restriction map
from the set of isomorphism classes of evaluation representations of
$M(X,\g)$ to the set of isomorphism classes of evaluation
representations of $M(X,\g)^\Gamma$ is surjective if $\g^x = \g$ for
all $x \in X$.
\end{prop}

\begin{proof}
The proof is immediate.
\end{proof}


\section{Classification of irreducible finite-dimensional representations} \label{sec:classification}

We consider a finite group $\Gamma$, acting on a finite-dimensional
Lie algebra $\g$ and an affine scheme $X$. We put $\frM =
M(X,\g)^\Gamma$.

\begin{lem} \label{lem:ideal-of-subring} {\rm (\cite[II, Corollaire~2 de la
Proposition~6, \S3.6]{bou:A})}
  Let $S$ be a subring of a ring $R$ and $I$ an ideal of $S$.  Then
  $R \otimes_S (S/I) \cong R/(RI)$.
\end{lem}

Note that $A^\Gamma$ is the coordinate ring of the quotient
$X{\gq}\Gamma$. For a point $\left[ x\right] \in X{\gq}\Gamma$, let
$\mathfrak{m}_{\left[ x\right]}$ denote the corresponding maximal
ideal of $A^\Gamma$ and define $A_{\left[ x\right]} = A
\otimes_{A^\Gamma} (A^\Gamma/\mathfrak{m}_{\left[ x\right]})$. We
will sometimes view $\left[ x\right]$ as an orbit in $X$.

\begin{prop} \label{prop:evaluation-factoring}
  Suppose $\frK$ is an ideal of $\frM$ such that the quotient
  algebra $\frM/\frK$ is finite-dimensional and simple.  Then there
  exists a point $x \in X_\rat$ such that the canonical epimorphism
  $\pi : \frM \to \frM/\frK$ factors through the evaluation map
  $\ev_x : \frM \to \g^x$.
\end{prop}

\begin{proof}
We have shown in Proposition~\ref{prop:fixed-point-alg-eval} that
there exists a character $\chi : A^\Gamma \to k$ such that $\pi(\al
f) = \pi(\al) \chi(f)$ holds for all $\al \in \frM$ and $f\in
A^\Gamma$. Let $\frn = \Ker \chi \in \Spec(A^\Gamma)$ be the
corresponding $k$-rational point. Temporarily viewing $\frM$ as a
$A^\Gamma$-module, it follows that $\frn$ annihilates $\frM/\frK$,
whence $\frn \frM \subseteq \frK$, and we can factor $\pi$ through
obvious maps:
\[
  \xymatrix
  { \frM \ar@{->>}[r]^(0.4)\pi \ar@{->>}[d] & \frM/\frK \\
       \frM/\frn \frM \ar@{->>}[ru] }
\]
Observe
\begin{equation} \label{thm:fac1}
 \frM/\frn\frM \cong \frM \ot_{A^\Gamma} (A^\Gamma/\frn)
   = (\g \ot _k A)^\Gamma \ot_{A^\Gamma} (A^\Gamma/\frn)
   = \big(\g \ot_k A  \ot_{A^\Gamma} (A^\Gamma/\frn) \big)^\Gamma,
\end{equation}
where in the last equality we used that $\Gamma$ acts trivially on
$A^\Gamma/\frn$. Also note, by Lemma~\ref{lem:ideal-of-subring},
\begin{equation} \label{thm:fac2}
   A \ot_{A^\Gamma} (A^\Gamma/\frn) \cong A / I \quad \hbox{for } I=A
\frn. \end{equation} Hence, putting these canonical isomorphisms
together, we get a new factorization:
$$\xymatrix
{ \frM \ar@{->>}[r]^\pi \ar@{->>}[d] & \frM/\frK \\
   \frM/\frn \frM  \ar[r]^<<<<\cong  &(\g \ot_k \big(A/I)\big)^\Gamma
         \ar@{->>}[u]_\psi }
$$
It now remains to show that $\psi$ factors through $\ev_x$ for an
appropriate $x\in X$. To this end, let $[\frn ] = \{ \fm \in \Spec A
:  \fm \cap A^\Gamma = \frn\}$. One knows (\cite[V, \S2]{Bou:ACb})
that $[\frn]$ is a non-empty set of $k$-rational points on which
$\Gamma$ acts transitively.  Let $[\frn] = \{ \fm_1, \ldots,
\fm_s\}$. We claim
\begin{equation} \label{thm:fac3}
   \sqrt{I} = \fm_1 \cap \dots \cap \fm_s.
\end{equation}
Indeed, any $\frp \in V(I) \subseteq \Spec A$ satisfies $\frp \cap
A^\Gamma = \frn$ so that $V(I) = [\frn]$.  Now (\ref{thm:fac3})
follows from \cite[II, \S2.6, Corollaire de la
Proposition~13]{Bou:ACa}.

We have an exact sequence of algebras and $\Gamma$-modules
\begin{equation} \label{eq:pro1}
  \textstyle
  0 \to \g\ot_k \big((\bigcap_i \fm_i)/I \big) \to \g \ot_k (A/I)
  \to \g \ot_k ( A /\bigcap_i \fm_i) \to 0,
\end{equation}
where
\begin{equation} \textstyle \label{eq:pro2}
  \g \ot_k(A / \bigcap_i \fm_i)
  \cong \bigoplus_i \big(\g \ot_k (A/\fm_i)) \
\end{equation}
since $A/\bigcap_i \fm_i \cong \bigoplus_i A/\fm_i$.  Now observe
that the summands on the right hand side of (\ref{eq:pro2}) are
permuted by the action of $\Gamma$. Thus, if we fix $x=\fm \in [
\frn ]$ we have
\[
  \textstyle
  \big( \g \ot_k (A / \bigcap_i \fm_i ) \big)^\Gamma
  \cong ( \g \ot_k (A/\fm))^{\Gamma_x} = \g^x.
\]
Therefore, taking $\Gamma$-invariants in (\ref{eq:pro1}) we get an
epimorphism
\[
  \zeta : \big( \g\ot_k (A/I)\big)^\Gamma
  \twoheadrightarrow \g^x
\]
with kernel $\big(\g\ot_k (\bigcap_i \fm_i)/I \big)^\Gamma$.

Any $\al \in \Ker \zeta$ is a finite sum $\al = \sum_j u_j \ot \bar
f_j$, where every $\bar f_j\in (\bigcap_i \fm_i)/I$ is nilpotent by
(\ref{thm:fac3}). The ideal $J$ of $\big(\frg \ot_k
(A/I)\big)^\Gamma$ generated by $\al$ is therefore nilpotent.  Since
$\frM/\frK$ is simple, it does not contain a non-zero nilpotent
ideal. Thus $\psi(J) = 0$. Therefore $\psi$ factors through $\zeta$
and we get the following commutative diagram:
\[
  \xymatrix{
  \mathfrak{M} \ar[rr]^\pi  \ar@{->>}[d]_p && \frM/\frK \\
  \big( \g \ot_k (A/I)\big)^\Gamma \ar@{->>}[rru]^\psi  \ar@{->>}[rr]_\zeta && \g^x \ar[u]
  }
\]
Since $\zeta \circ p = \ev_x$, the result follows.
\end{proof}

\begin{rem} \label{rem:generalization}
Proposition~\ref{prop:evaluation-factoring} remains true for $k$ an
arbitrary algebraically closed field whose characteristic does not
divide the order of $\Gamma$ and for $\g$ an arbitrary
finite-dimensional (not necessarily Lie or associative) algebra.
\end{rem}

\begin{cor} \label{cor:small-reps} \mbox{}
\begin{enumerate}
  \item \label{cor-item1:small-reps}
  Suppose $\varphi : \mathfrak{M} \to \End_k V$ is an irreducible finite-dimensional
  representation such that $\mathfrak{M}/\ker
  \varphi$ is a semisimple Lie algebra.  Then $\varphi$ is an irreducible finite-dimensional
  evaluation representation.
  \item \label{cor-item2:small-reps} The irreducible finite-dimensional representations of\/ $\frM$ are precisely the
  representations of the form $\la \ot \varphi$ where $\la \in
  (\frM/[\frM,\frM])^*$ and $\varphi$ is an irreducible finite-dimensional evaluation
  representation with $\frM/\Ker \varphi$ semisimple. The factors $\la$
  and $\varphi$ are uniquely determined.
\end{enumerate}
\end{cor}

\begin{proof}
  Part \eqref{cor-item1:small-reps} follows from
  Proposition~\ref{prop:evaluation-factoring} and the results of
  Section~\ref{sec:review}.  Part \eqref{cor-item2:small-reps} is
  then a simple application of Lemma~\ref{lem:deco}.
\end{proof}

We now state our main theorem which gives a classification of the
irreducible finite-dimensional representations of an arbitrary
equivariant map algebra.  Recall the definitions of $\mathcal{E}$
(Definition~\ref{def:E}) and $\scS$ (Definition~\ref{def:eval-rep}).

\begin{theo} \label{thm:fd-reps=eval-reps+linear-form}
Suppose $\Gamma$ is a finite group acting on an affine scheme $X$
and a finite-dimensional Lie algebra $\g$.  Then the map
\[
  (\frM/[\frM,\frM])^* \times \mathcal{E} \to \scS,\quad
  (\lambda, \Psi) \mapsto \lambda \otimes \ev_\Psi,\qquad
  \lambda \in (\frM/[\frM,\frM])^*,\quad \Psi \in \mathcal{E},
\]
is surjective.  In particular, all irreducible finite-dimensional
representations of $\frM$ are tensor products of an evaluation
representation and a one-dimensional representation.

Furthermore, we have that $\lambda \otimes \ev_\Psi = \lambda'
\otimes \ev_{\Psi'}$ if and only if there exists $\Phi \in
\mathcal{E}$ such that $\dim \ev_\Phi=1$, $\lambda' = \lambda -
\ev_\Phi$ and $\ev_{\Psi'} = \ev_{\Psi \otimes \Phi}$.  Here $\Psi
\otimes \Phi \in \mathcal{E}$ is given by $(\Psi \otimes \Phi)(x) =
\Psi(x) \otimes \Phi(x)$, where $\Psi(x)$ (respectively $\Phi(x)$)
is the one-dimensional trivial representation if $x \not \in \supp
\Psi$ (respectively $x \not \in \supp \Phi$). In particular, the
restriction of the map $(\lambda, \Psi) \mapsto \lambda \otimes
\ev_{\Psi}$ to either factor (times the zero element of the other)
is injective.
\end{theo}

\begin{proof}
This follows from Corollary~\ref{cor:small-reps} and
Lemma~\ref{lem:deco}.
\end{proof}

\begin{rem}
Since $\lambda, \mu \in (\frM/[\frM,\frM])^*$ are isomorphic as
representations if and only if they are equal as linear functions,
in Theorem~\ref{thm:fd-reps=eval-reps+linear-form} we have
identified elements of $(\frM/[\frM,\frM])^*$ with isomorphism
classes of representations.
\end{rem}

\begin{rem}
Note that the evaluation representations of $M(\Spec A,\g)^\Gamma$
are the same as the evaluation representations of $M(\Spec (A/\rad
A),\g)^\Gamma$ since one evaluates at rational points.  However, the
one-dimensional representations of these two Lie algebras can be
different in general.  Thus, we do not assume that the scheme $X$ is
reduced.
\end{rem}

\begin{cor}\label{cor:perfectreps}
If $\frM$ is perfect, then the map $\Psi \mapsto \ev_\Psi$ is a
bijection between $\mathcal{E}$ and $\scS$.  In particular, this is
true if any one of the following conditions holds:
\begin{enumerate}
  \item $[\g^\Gamma,\g] = \g$, \label{perfect-cond1}
  \item $\g$ is simple, $\g^\Gamma \ne \{0\}$ is perfect and acts on $\g_\Gamma$ without
   a trivial non-zero submodule,
   or \label{perfect-cond2}
  \item $\Gamma$ acts on $\g$ by diagram automorphisms.
  \label{perfect-cond3}
\end{enumerate}
\end{cor}

\begin{proof}
If $\frM$ is perfect, then $[\frM,\frM]=\frM$, and the first
statement follows immediately from
Theorem~\ref{thm:fd-reps=eval-reps+linear-form}.
Conditions~\eqref{perfect-cond1} or~\eqref{perfect-cond2} imply that
$\frM$ is perfect by Lemma~\ref{lem:perfrm}. It remains to show that
condition~\eqref{perfect-cond3} implies that $\g^\Gamma$ is perfect.
It suffices to consider the case where $\g$ is simple.  If $\Gamma$
acts on $\g$ by diagram automorphisms, then there are two
possibilities: either $\Gamma$ is a cyclic group generated by a
single diagram automorphism or $\g$ is of type $D_4$ and $\Gamma
\cong S_3$.  If $\Gamma$ is generated by a single diagram
automorphism, it is well known that $\g^\Gamma$ is a simple Lie
algebra and hence perfect (see \cite[\S8.2]{Kac90}). The case
$\Gamma \cong S_3$ was described in Example~\ref{eg:nonabelian},
where it was shown that $\g^\Gamma$ is simple as well.
\end{proof}

\begin{rem}
Note that the three conditions in Corollary~\ref{cor:perfectreps}
depend only on the action of $\Gamma$ on $\g$ and not on the scheme
$X$ or its $\Gamma$-action.
\end{rem}

\begin{rem}[Untwisted map algebras] \label{rem:untwisted}
If $\Gamma$ is trivial (or, more generally, acts trivially on $\g$),
we have $\g^\Gamma = \g$.  Thus the map $\mathcal{E} \to \scS$,
$\Psi \mapsto \ev_\Psi$, is a bijection if and only if $\g$ is
perfect. In the case when $\Gamma$ is trivial, $\g$ is a
finite-dimensional simple Lie algebra, and $A$ is finitely
generated, a similar statement has recently been made in \cite{CFK}.
\end{rem}

\begin{cor} \label{comp-red}
Suppose that all irreducible finite-dimensional representations of
$\frM$ are evaluation representations (e.g. $\frM$ is perfect) and
that all $\g^x$, $x\in X_\rat$, are semisimple. Then a
finite-dimensional $\frM$-module $V$ is completely reducible if and
only if there exists $\bx \in X_n$ for some $n\in \NN$, $n>0$, such
that $\Ker \ev_\bx \subseteq \Ann_\frM V := \{\alpha \in \frM :
\alpha \cdot v =0 \ \forall\ v \in V\}$.
\end{cor}

For the case of the current algebra $\frM=\g \ot \C[t]$,  this
corollary is proven in \cite[Prop.~3.9(iii)]{cha-gre}.

\begin{proof} Let $V$ be a completely reducible $\frM$-module, hence a finite direct
sum of irreducible finite-dimensional representations $V^{(i)}$. By
assumption, every $V^{(i)}$ is an evaluation representation, given
by some $\bx^{(i)} \in X_{n_i}$. Then $\Ann_\frM V = \bigcap_i
\Ann_\frM V^{(i)} \supseteq \bigcap_i \Ker \ev_{\bx^{(i)}}= \Ker
\ev_{\mathbf{y}}$ for ${\mathbf{y}} = \bigcup_i \bx^{(i)}$.  Since
$\Ker \ev_x=\Ker \ev_{g\cdot x}$, we can replace ${\mathbf{y}}$ by
some $\bx \in X_n$ satisfying $\Ker \ev_{\mathbf{y}} = \Ker
\ev_\bx$.

Conversely, if $\Ker \ev_\bx \subseteq \Ann_\frM (V)$, then the
representation of $\frM$ on $V$ factors through the semisimple Lie
algebra $\bigoplus_{x\in \bx} \g^x$ and is therefore completely
reducible.  \end{proof}

While Theorem~\ref{thm:fd-reps=eval-reps+linear-form} classifies all
the irreducible finite-dimensional representations of an arbitrary
equivariant map algebra, in case $\frM$ is not perfect it leaves
open the possibility that not all irreducible finite-dimensional
representations are evaluation representations.  We see that $\frM$
has irreducible finite-dimensional representations that are not
evaluation representations precisely when it has one-dimensional
representations that are not evaluation representations.  We
therefore turn our attention to one-dimensional evaluation
representations.  Let
\[
  \tilde X = \{x \in X_\rat : [\g^x, \g^x] \ne \g^x\}.
\]
Note that $\tilde X$ is a $\Gamma$-invariant subset of $X$ (i.e.
$\tilde X$ is a union of $\Gamma$-orbits).

\begin{lem} \label{lem:one-dim-evals}
If $\ev_\Psi$ is (the isomorphism class of) a one-dimensional
representation, then $\supp \Psi \subseteq \tilde X$.
\end{lem}

\begin{proof}
This follows easily from the fact that for $x \in X \setminus \tilde
X$ we have that $\g^x$ is perfect and thus the one-dimensional
representations of $\g^x$ are trivial.
\end{proof}

Let
\[
  \frM^d = \{\alpha \in \frM : \alpha(x) \in [\g^x,\g^x]\ \forall\ x
  \in X_\rat\} = \{\alpha \in \frM : \alpha(x) \in [\g^x,\g^x]\ \forall\
  x \in \tilde X\}.
\]
Then it is easy to see that $[\frM,\frM] \subseteq \frM^d$.  The
proof of the following lemma is straightforward.

\begin{lem}
The Lie algebra $\frM$ is perfect if and only if $\frM^d =
[\frM,\frM]$ and $\tilde X = \emptyset$.
\end{lem}

Now assume that $|\tilde X| < \infty$. Let $\mathbf{x}$ be a set of
representatives of the $\Gamma$-orbits comprising $\tilde X$ and
consider the composition
\begin{equation} \label{eq:tildeX-composition}
  \xymatrix{
  \frM \ar@{->>}[r]^(.35){\ev_{\mathbf{x}}} & \bigoplus_{x \in \mathbf{x}}
  \g^x \ar@{->>}[r]^(.3){\pi} & \bigoplus_{x \in \mathbf{x}}
  \frz^x,\quad \frz^x := \g^x/[\g^x,\g^x],
  }
\end{equation}
where the $x$-component of $\pi$ is the canonical projection $\g^x
\to \g^x/[\g^x,\g^x]$.  If $\g$ is reductive, then so is every
$\g^x$, and we can identify $\frz^x$ with the center $Z(\g^x)$ of
$\g^x$.  However, we will not assume that $\g$ is reductive.  The
kernel of \eqref{eq:tildeX-composition} is precisely $\frM^d$, and
thus the composition factors through $\frM/[\frM, \frM]$, yielding
the following commutative diagram:
\begin{equation} \label{eq:center-surjection}
  \xymatrix{
  \frM \ar@{->>}[r]^(.35){\ev_{\mathbf{x}}} \ar@{->>}[dr] & \bigoplus_{x \in \mathbf{x}}
  \g^x \ar@{->>}[r]^{\pi} & \bigoplus_{x \in \mathbf{x}}
  \frz^x \\
  & \frM/[\frM,\frM] \ar@{->>}[ur]_{\gamma} &
  }
\end{equation}
We then have an isomorphism of vector spaces,
\begin{equation} \label{eq:linear-form-decomp} \textstyle
  (\frM/[\frM,\frM])^* \cong (\ker \gamma)^* \oplus \left(\bigoplus_{x \in \mathbf{x}}
  \frz^x \right)^*.
\end{equation}

\begin{prop} \label{prop:one-dim-eval-reps}
If $|\tilde X| < \infty$ and $\mathbf{x}$ is a set of
representatives of the $\Gamma$-orbits comprising $\tilde X$, then
there is a natural identification
\[ \textstyle
  \left(\bigoplus_{x \in \mathbf{x}} \frz^x \right)^* \cong \{\ev_\Psi
  : \Psi \in \mathcal{E},\ \dim \ev_\Psi = 1 \}.
\]
\end{prop}

\begin{proof}
Choose $\lambda \in \left(\bigoplus_{x \in \mathbf{x}} \frz^x
\right)^*$. To $\lambda$ we associate the evaluation representation
\[
  \xymatrix{
  \frM \ar@{->>}[r]^(.35){\ev_{\mathbf{x}}} & \bigoplus_{x \in \mathbf{x}}
  \g^x \ar@{->>}[r]^(.47){\pi} & \bigoplus_{x \in \mathbf{x}}
  \frz^x \ar[r]^(.6){\lambda} & k.
  }
\]
By Lemma~\ref{lem:one-dim-evals}, this gives the desired bijective
correspondence.
\end{proof}

We can now refine Theorem~\ref{thm:fd-reps=eval-reps+linear-form} as
follows.

\begin{theo} \label{thm:finite-tilde-X}
Suppose $\Gamma$ is a finite group acting on an affine scheme $X$
and a finite-dimensional Lie algebra $\g$, and assume that $|\tilde
X| < \infty$. If $\gamma$ is defined as in
\eqref{eq:center-surjection}, then the map
\[
  (\lambda, \Psi) \mapsto \lambda \otimes \ev_\Psi,\quad \lambda \in
  (\ker \gamma)^*,\quad \Psi \in \mathcal{E}
\]
is a bijection between $(\ker \gamma)^* \times \mathcal{E}$ and
$\scS$.
\end{theo}

\begin{proof}
This follows from Theorem~\ref{thm:fd-reps=eval-reps+linear-form},
\eqref{eq:linear-form-decomp} and
Proposition~\ref{prop:one-dim-eval-reps}.
\end{proof}

\begin{cor}\label{cor:non-ss-reps}
Assume $|\tilde X| < \infty$.  Then $[\frM,\frM]=\frM^d$ if and only
if all irreducible finite-dimensional representations are evaluation
representations.
\end{cor}

\begin{proof}
By Theorem~\ref{thm:finite-tilde-X}, all irreducible
finite-dimensional representations are evaluation representations if
and only if $\gamma$ is injective (and hence an isomorphism, since
it is surjective).  Then the result follows from the commutative
diagram \eqref{eq:center-surjection} since the kernel of
\eqref{eq:tildeX-composition} is $\frM^d$.
\end{proof}

\begin{rem}
Note that if $\g$ is perfect and $\Gamma$ acts on $X$ in such a way
that there are only a finite number of points of $X$ that have a
non-trivial stabilizer, then $|\tilde X| < \infty$, and so the
hypotheses of Theorem \ref{thm:finite-tilde-X} are satisfied.
\end{rem}

\begin{rem}
In Section~\ref{subsec:onsager} we will see that the Onsager algebra
is an equivariant map algebra which is not perfect but for which
$\gamma$ is injective, and thus  all irreducible finite-dimensional
representations are nonetheless evaluation representations.
\end{rem}

Having considered the case when all irreducible finite-dimensional
representations are evaluation representations, we now examine the
opposite situation: Equivariant map algebras for which there exist
irreducible finite-dimensional representations that are not
evaluation representations.

\begin{prop} \label{prop:infinite-tildeX}
Suppose $X$ is a Noetherian affine scheme and $\tilde X$ is
infinite. Then\/ $\frM$ has a one-dimensional representation that is
not an evaluation representation.
\end{prop}

\begin{proof}
We first set up some notation for one-dimensional evaluation
representations.  Let $\bx \in X_n$ and let $\rho_x : \frg^x \to
\End_k(V_x)$, $x \in \bx$, be representations such that $\ev_\bx
(\rho_x)_{x \in \bx}$ is a one-dimensional representation.
Necessarily $\dim V_x = 1$, say $V_x = k v_x$, so $V= \bigotimes_{x
\in \bx} V_x = k \bv$ for $\bv = \bigotimes_{x \in \bx} v_x$. We can
assume that all $\rho_x \ne 0$, whence $\bx \subseteq \tilde X$. Let
$\tilde \rho_x \in (\g^x)^*$ be defined by $\rho_x(u) (v_x) = \tilde
\rho_x(u) v_x$ for $u\in \g^x$. For $\al= \sum_i u_i \ot f_i\in
\frM$ we then have
\[
  \big( (\ev_\bx (\rho_x)_{x \in \bx})(\al )\big) (\bv)
  = \left(\sum_{x \in \bx,\, i} \tilde \rho_x(u_i) f_i(x)\right)  \bv.
\]
Suppose that there exist $\al \in \frM$ and $f \in A^\Gamma$ such
that
\[
   \{ \al f^m : m\in \NN\} \mbox{ is linearly independent and }
  \Span \{ \al f^m : m\in \NN\} \cap [\frM,\frM] =
  0. \eqno{(*)}
\]
Then there exists $\la\in \frM^*$ such that $\la([\frM, \frM])=0$
and $\la(\al f^m)\in k$, $m \ge 1$, is a root of an irreducible
rational polynomial $p_m$ of degree $m$, for example $p_m(z) = z^m
-2$. Now suppose that the corresponding one-dimensional
representation is an evaluation representation. Writing $\al=\sum_i
u_i \ot f_i$ with $u_i \in \g$ and $f_i \in A$, we get from the
equation above
\[
   \la(\al f^m) = \sum_{x \in \bx,\, i} \tilde \rho_x(u_i) (f_i f^m)(x)
      = \sum_{x \in \bx,\, i} \tilde \rho_x(u_i) f_i(x) f(x)^m.
\]
But this is a contradiction since the elements $\tilde \rho_x(u_i)
f_i(x) f(x)^m\in k$ all lie in the $\mathbb{Q}$-subalgebra of $k$
generated by the finitely many elements $\tilde \rho_x(u_i), f_i(x),
f(x)$ of $k$, while the elements $\la(\alpha f^m)$, $m\in \NN$, do
not lie in such a subalgebra.

We will now construct $\al \in \frM$ and $f\in A^\Gamma$ satisfying
($*$). For a subgroup $H$ of $\Gamma$, let
\[
  X_H = \{x \in X_\rat : \Gamma_x = H\}.
\] Since
\[
  \tilde X = \bigcup_{H\, :\, \g^H \ne [\g^H,\g^H]} X_H,
\]
there exists a subgroup $H$ of $\Gamma$, with $\g^H \ne
[\g^H,\g^H]$, such that $X_H$ is infinite.  Observe that $X_H = X^H
\setminus \bigcup_{K\supsetneq H} X^K$ is open in the closed subset
$X^H$ and that $X^H$ is a Noetherian affine scheme since $X$ is.
Because $X^H$ has only finitely many irreducible components, there
exists an irreducible component $Y$ of $X^H$ such that $\tilde Y = Y
\cap X_H$ is infinite.

Let $R : A_X^\Gamma \to A_Y$ be the restriction map. Choose an
infinite set $\{y_1, y_2, \dots\}$ of points of $Y$, no two of which
are in the same $\Gamma$-orbit. It follows as in the proof of
Corollary~\ref{cor:twisted-surjectivity} that for all $j \in
\mathbb{N}$ there exists $f_j \in A_X^\Gamma$ such that $f_j(y_j)\ne
0$ and $f_j(y_i) = 0$ for $i<j$. Since the set $\{R(f_j) : j\in
\NN\}$ is linearly independent, the image of $R$ is
infinite-dimensional. Because $A_X^\Gamma$ is finitely generated, so
is the image of $R$. Therefore, by the Noether normalization lemma,
this image contains an element $f_Y$ such that the set
$\{1,f_Y,f_Y^2,\dots,\}$ is linearly independent. Choose $f \in
R^{-1}(f_Y)$.  It follows that $\{1,f,f^2,\dots,\}$ is also linearly
independent. Since $\tilde Y$ is open in the irreducible $Y$ we can
choose $\tilde y \in \tilde Y$ such that $f(\tilde y) \ne 0$.

Let $\{u_i\}_{i=1}^l$ be a basis of $[\g^H, \g^H]$, and complete it
to a basis $\{u_i\}_{i=1}^m$ of $\g^H$.  Then complete this to a
basis $\{u_i\}_{i=1}^n$ of $\g$. Again by
Corollary~\ref{cor:twisted-surjectivity} there exists $\al\in \frM$
such that $\al(\tilde y) = u_m$. Observe that $\al_Y = \al|_Y$ can
be written in the form $\al_Y = \sum_{i=1}^m u_i \otimes f_i$ for
some $f_i \in A_Y$. Since $\al(\tilde y) = u_m$, we have $f_m(\tilde
y) = 1$. Therefore $f_m(y) \ne 0$ for all $y$ in some dense subset
of $Y$. Now suppose
\[
  \sum_{j=0}^\infty d_j \alpha f^j \in [\frM,\frM]
\]
for some $d_j \in k$ with $d_j=0$ for all but finitely many $j$. We
then have
\[
  \left. \left( \sum_{j=0}^\infty d_j \alpha f^j \right) \right|_Y
  = \alpha_Y \sum_{j=0}^\infty d_j f_Y^j = \left( \sum_{i=1}^m u_i
  \otimes f_i \right) \sum_{j=0}^\infty d_j f_Y^j = \sum_{i=1}^m u_i
  \otimes \left( f_i \sum_{j=0}^\infty d_j f_Y^j \right).
\]
Since $f_m(y)\ne 0$ for $y$ in a dense subset of $Y$, we must have
$\sum_{j=0}^\infty d_j f_Y^j=0$. Because $\{1, f_Y, f_Y^2,\dots\}$
is linearly independent, we must have $d_j = 0$ for all $j$.
Therefore
\[
  \Span \{ \al f^m : m\in \NN\} \cap [\frM,\frM] =
0.
\]
Now observe that the preceding argument also shows that $\{\alpha,
\alpha f, \alpha f^2, \dots\}$ is linearly independent: Suppose
$\sum_{j=1}^\infty c_j \alpha f^j=0$ for some $c_j \in k$ with $c_j
= 0$ for all but finitely many $j$. Then, since $0\in [\frM, \frM]$,
all $c_j = 0$ by what we have just shown. Thus ($*$) holds, finishing
the proof of the proposition. \end{proof}

\begin{example} \label{ex:infinite-tildeX} We give an example in which
the assumptions of Proposition~\ref{prop:infinite-tildeX} are
fulfilled. Let $\Gamma=\{1, \si\}$ be the group of order two acting
on $\g=\lsl_2(k)$ by the Chevalley involution with respect to some
$\lsl_2$-triple and let $X$ be the affine space $k^2$ with $\si$
acting on $X$ by fixing the first coordinate of points in $X$ while
multiplying the second by $-1$.  For points $(x_1,x_2) \in X$ with
$x_2 \ne 0$, the isotropy subalgebra $\g^x = \g$, while for $(x_1,0)
\in X$, the subalgebra $\g^x$ is the fixed point subalgebra of
$\sigma$, which is one-dimensional.  Therefore $\tilde X = \{(x_1,0)
\in X\}$ is infinite.
\end{example}

Proposition~\ref{prop:infinite-tildeX} says that when $X$ is an
affine variety, a necessary condition for all irreducible
finite-dimensional representations to be evaluation representations
is that $\tilde X$ be finite.  We now show that this condition is
not sufficient.

\begin{example}
Let $\g = \mathfrak{sl}_2(k)$ and
\[
  X = Z(y^2-x^3) =\{(y,x) : y^2=x^3\} \subseteq k^2,
\]
an affine variety.  Then $A = k[y,x]/(y^2-x^3)$.  Let $\Gamma =
\langle \sigma \rangle = \Z_2$ act on $k^2$ by $\sigma \cdot y = -y$
and $\sigma \cdot x = x$.  Since this action fixes $y^2-x^3$, we
have an induced action of $\Gamma$ on $X$ and the only fixed point
is the origin. In particular, $\tilde X$ only contains the origin
and thus is finite.  We let $\sigma$ act on $\g$ by a Chevalley
involution. We have
\[
  A_1 = y k[y^2,x]/(y^2-x^3),
\]
and so
\[
  A_1^2 = y^2 k[y^2,x]/(y^2-x^3) \cong x^3 k[x],
\]
while
\[
  A_0 = A^\Gamma = k[y^2,x]/(y^2-x^3) \cong k[x].
\]
Recall that $\g_0$ is one-dimensional.  Then
\[
  [\frM,\frM] = ([\g_0,\g_0]\otimes A_0 + \g_0 \otimes
  A_1^2) \oplus (\g_1 \otimes A_1) = \left( \g_0 \otimes x^3
  k[x] \right) \oplus \left( \g_1 \otimes A_1 \right),
\]
while
\[
  \frM^d = \left( \g_0 \otimes x k[x] \right) \oplus \left(
  \g_1 \otimes A_1 \right).
\]
Therefore
\[
  \frM^d/[\frM,\frM] \cong x k[x]/(x^3),
\]
and by Theorem~\ref{thm:finite-tilde-X}, $\frM$ has irreducible
finite-dimensional representations that are not evaluation
representations.
\end{example}


\section{Applications} \label{sec:applications}

In this section we use our classification to describe the
irreducible finite-dimensional representations of certain
equivariant map algebras. The classification of these
representations for the multiloop, tetrahedron, or Onsager algebra
$\On(\mathfrak{sl}_2)$ by the results obtained in
Section~\ref{sec:classification} provide a simplified and unified
interpretation of results previously obtained.  For example, the
classification of the irreducible finite-dimensional representations
of $L^\sigma(\lie g)$ (as found in \cite{CPweyl} and \cite{CFS}) via
Drinfeld polynomials requires two distinct treatments for the
untwisted ($\sigma = \Id$) and twisted ($\sigma \neq \Id$) cases --
for these twisted cases, the twisted loops $L^\sigma(\lie g)$
with $\lie g$ of type $A_{2n}$ require special attention.  The
classification resulting from our approach, however, is uniform.  As
we will see, the identification of isomorphism classes of
representations with equivariant maps $X_\rat \to \mathcal{R}_X$
also provides a simple explanation for many of the technical
conditions appearing in previous classifications.

We first note that from Remark \ref{rem:untwisted} we immediately
obtain the classification of irreducible finite-dimensional
representations of the current algebras (Example~\ref{eg:current}),
of the untwisted loop and multiloop algebras
(Example~\ref{eg:multiloop}), and of the $n$-point algebras
$M(X,\g)$, where $X = \mathbb{P}^1 \backslash \left\{ c_1, \ldots,
c_n \right\}$. This includes the tetrahedron algebra, which is
isomorphic to the three-point $\mathfrak{sl}_2$ loop algebra
$\mathfrak{sl}_2 \otimes \mathbb{C}\left[ t, t^{-1},
(t-1)^{-1}\right]$ (Example~\ref{eg:threepoint}), and we recover the
classification found in \cite{Ha}.  In particular, we easily recover
in all of these cases the fact that for $\bx = \{x_1, \ldots, x_l\}
\subseteq X$ and irreducible representations $V_1, \ldots, V_l$ of
$\g$, the evaluation representation $\ev_{\bx}(\otimes V_i)$ is
irreducible if and only if $x_i \neq x_j$ for $i \neq j$.

\subsection{Multiloop algebras}

If $\M = M(\g,\sigma_1,\dots,\sigma_n,m_1,\dots,m_n)$ is a multiloop
algebra (Example~\ref{eg:twistedmultiloop}), then $\M$ is perfect
 ($\M$ is an iterated loop algebra; see {\cite[Lemma~4.9]{abp}}).
Therefore, by Corollary~\ref{cor:perfectreps} we have the following
classification:

\begin{cor}\label{cor:multiloop-reps}
The map $\mathcal{E} \to \scS$, $\Psi \mapsto \ev_\Psi$, is a
bijection. In particular, all irreducible finite-dimensional
representations are evaluation representations.
\end{cor}

The irreducible finite-dimensional representations of an arbitrary
multiloop algebra have been discussed in \cite{Bat} and \cite{Lau}.
With Corollary \ref{cor:multiloop-reps} we recover the recent
results in \cite{Lau}, which subsume all previous classifications of
irreducible finite-dimensional representations of loop algebras. We
note that these previous classifications involved some rather
complicated algebraic conditions on points of evaluation (see, for
example, \cite[Theorem~5.7]{Lau}).  However, in the approach of
Theorem~\ref{thm:fd-reps=eval-reps+linear-form}, such conditions are
not necessary.  In fact, we see that the presence of these algebraic
conditions arises from the description of evaluation representations
in terms of individual points rather than as equivariant maps (i.e.
elements of $\mathcal{E}$).  For instance, if $\M =
M(\g,\sigma_1,\dots,\sigma_n,m_1,\dots,m_n)$
 is an arbitrary multiloop algebra and
$\Psi \in \mathcal{E}$, then $\ev_{\Psi}$ is the isomorphism class
$\ev_{\bx}(\Psi(x_i))_{i=1}^l$, where $\bx = \{x_1,\dots,x_l\} \in
X_l$ contains one element from each $\Gamma$-orbit in $\supp \Psi$
and this class is independent of the choice of $\bx$ (Definition
\ref{def:evaluation-rep}).  It is immediate that such an $\bx$ must
satisfy the condition that $x_1^m, \ldots, x_l^m$ (where $x_i^m =
(x_{i1}^{m_1}, \dots, x_{in}^{m_n})$ for $x_i=(x_{i1}, \dots,
x_{in}$)) are pairwise distinct in $(k^\times)^n$. We therefore
recover the conditions on the points $x_i$ found in
\cite[Theorem~5.7]{Lau} which are necessary and sufficient for
$\ev_{x_1, \ldots, x_l}(\otimes V_i)$ to be irreducible.  The other
conditions found there are similarly explained.

\subsection{Connections to Drinfeld polynomials}

In \cite{CPweyl}, \cite{CM}, \cite{CFS} and \cite{S}, the
isomorphism classes of irreducible finite-dimensional
representations of loop algebras $L(\g)$, $L^\sigma(\g)$ are
parameterized by certain collections of polynomials, sometimes
referred to as \emph{Drinfeld polynomials}.  Here we explain the
relationship between this parametrization and ours.  For better
comparison with the existing literature, we assume in this
subsection that $k=\mathbb{C}$.

Denote by $\calP$ the set of of $n$-tuples of polynomials with
constant term 1:
\[
  \calP = \left\{ \bpi = \left( \pi_1(u), \ldots, \pi_n(u) \right) :
  \pi_i \in \mathbb{C}\left[ u \right], \; \; \pi_i(0) = 1 \right\}.
\]
Then the set $\calP$ is in bijective correspondence with the
isomorphism classes of irreducible finite-dimensional
representations of $L(\g)$ (\cite[Proposition~2.1]{CPweyl}).  To the
element $\bpi \in \calP$ we associate an irreducible representation
$V(\bpi)$ (the construction of $V(\bpi)$ is given in \cite{CPweyl}).
We describe this correspondence.

Fix a simple finite-dimensional Lie algebra $\g$, denote by $n$ its
rank, and fix a Cartan decomposition $\g = \lie n^- \oplus \lie h
\oplus \lie n^+$ with Cartan subalgebra $\lie h = \bigoplus_{i=1}^n
\mathbb{C} h_i \subseteq \g$ and weight lattice $P  = \bigoplus_{i=1}^n
\mathbb{Z} \omega_i$, with fundamental weights $\left\{ \omega_i
\right\}_{i=1}^n$, $\omega_i(h_j) = \delta_{ij}$.

Let $\bpi = (\pi_1, \ldots, \pi_n) \in \calP$ and $\left\{ x_i
\right\}_{i=1}^l = \bigcup_{j=1}^n \left\{ z \in \mathbb{C}^\times :
\pi_j(z^{-1}) = 0 \right\}$.  Then each $\pi_j$ can be written
uniquely in the form
\[
  \pi_j(u) = \prod_{i=1}^l(1 - x_i u)^{N_{ij}}, \; \; \; N_{ij} \in \mathbb{N}.
\]
Let $\mathbf{x} = \{x_1, \ldots, x_l\}$.  For $i = 1, \ldots, l$,
define $\lambda_i \in P^+$ by $\lambda_i(h_j) = N_{ij}$, and let
$\rho_i : \g \rightarrow \End_k(V(\lambda_i))$ be the corresponding
irreducible finite-dimensional representation of $\g$.  Then
$V(\bpi)$ is isomorphic as an $L(\g)$-module to the evaluation
representation
\[ \textstyle
  \ev_{\mathbf{x}}(\rho_i)_{i=1}^l : L(\g) \stackrel{\ev_{\mathbf{x}}}{\longrightarrow}
  \g^{\oplus l} \xrightarrow{\bigotimes_{i=1}^l \rho_i}
  \End_k \left( \bigotimes_{i=1}^l V_i \right).
\]
To produce an element $\bpi \in \calP$ from an irreducible
representation $V$ of $L(\g)$, we first find an evaluation
representation $\ev_{\mathbf{x}}(\rho_i)_{i=1}^l : L(\g)
\longrightarrow \End_k( \bigotimes_{i=1}^l V(\lambda_i))$ isomorphic to
$V$ (\cite[Theorem~2.14]{R} or Corollary~\ref{cor:multiloop-reps}).
Next, for $i = 1, \ldots, l$, we define elements $ \bpi_{\lambda_i,
x_i} \in \calP$ by
\[
  \bpi_{\lambda_i, x_i} = \left( (1 - x_i u)^{\lambda_i(h_1)}, \ldots,
(1 - x_i u)^{\lambda_i(h_n)} \right)
\]
and define $\bpi  =
\prod_{i=1}^l \bpi_{\lambda_i, x_i}$, where multiplication of
$n$-tuples of polynomials occurs componentwise.

Given an element $\bpi \in \mathcal{P}$, we can uniquely decompose
$\bpi = \prod_{i=1}^l \bpi_{\lambda_i, x_i}$, $x_i \neq x_j$, and we
define
\[
  \Psi_{\bpi}
  := \left\{ x_i \mapsto \left[ V(\lambda_i) \right] \right\} \in
  \mathcal{E}.
\]
Then $V(\bpi)$ is a representative of $\ev_{\Psi_{\bpi}}$.

In \cite{CFS}, there is a similar parametrization of the irreducible
finite-dimensional representations of $L^\sigma(\g)$, where $\sigma$
is a non-trivial diagram automorphism of $\g$, but in this case the
bijective correspondence is between isomorphism classes of
irreducible finite-dimensional $L^\sigma(\g)$-modules and the set
$\mathcal{P}^\sigma$ of $m$-tuples of polynomials ${\mbox {\boldmath
$\pi$}^\sigma} = (\pi_1, \ldots, \pi_m)$, $\pi_i(0)=1$, where $m$ is
the rank of the fixed-point subalgebra $\g_0 \subseteq \g$. One
feature of this classification is the fact that every irreducible
finite-dimensional $L^\sigma(\g)$-module is the restriction of an
irreducible finite-dimensional $L(\g)$-module (see
\cite[Theorem~2]{CFS}). This fact follows immediately from
Proposition~\ref{prop:lifting} once we note that in the setup of
multiloop algebras, the action of $\Gamma$ on $X=\C^\times$ is via
multiplication by roots of unity and hence is free.  Thus $\g^{x} =
\g^{\Gamma_x} = \g^{\left\{ \Id \right\}} = \g$ for all $x \in X$.
Of course, the approach of the current paper yields an enumeration
by elements of $\mathcal{E}$.  The induced identification of
$\mathcal{E}$ with $\mathcal{P}^\sigma$ is somewhat technical and
will not be described here, but can found in \cite{Ssurvey}.

\subsection{The generalized Onsager algebra} \label{subsec:onsager}

Our results also provide a classification of the irreducible
finite-dimensional representations of the generalized Onsager
algebra $\On(\g)$ introduced in Example~\ref{eg:onsager} (in fact,
for the more general equivariant map algebra of
Example~\ref{eg:general-involution}). For $\g \ne \lsl_2$ this
classification was previously unknown.

\begin{prop}\label{prop:onsager-reps}
Let $\g$ be a simple Lie algebra, $X=\Spec k[t^{\pm 1}]$, and
$\Gamma=\{1, \sigma\}$ be a group of order two.  Suppose $\sigma$
acts on $X$ by $\sigma \cdot x = x^{-1}$, $x \in X$, and on $\g$ by
an automorphism of order two.  Then the map $\mathcal{E} \to \scS$,
$\Psi \mapsto \ev_\Psi$, is a bijection. In particular, all
irreducible finite-dimensional representations are evaluation
representations. In particular, this is true for the generalized
Onsager algebra $\On(\g)$.
\end{prop}

\begin{proof}
Recall that $\g_0 = \g^\Gamma$ is either semisimple or has a
one-dimensional center.  In the case when $\g_0$ is semisimple, the
result follows from Corollary \ref{cor:perfectreps} and
\eqref{eq:quotient-general-involution}.  We thus assume that
$Z(\g_0) \cong \g_0/[\g_0,\g_0]$ is one-dimensional.
By~\eqref{eq:quotient-general-involution}, we have $\frM/[\frM,\frM]
\cong A_0/A_1^2$, a Lie algebra with trivial Lie bracket. Now, $A_0
= k[t+t^{-1}]$ and $A_1=(t-t^{-1})A_0$. Thus, setting $z=t+t^{-1}$,
we have
\[
  A_0/A_1^2 \cong k[z]/\langle z^2-4 \rangle,
\]
which is a two-dimensional vector space.  The points 1 and $-1$ are
each $\Gamma$-fixed points, and so we must take $\mathbf{x} = \{\pm
1\}$ in \eqref{eq:center-surjection}.  Therefore $\bigoplus_{x \in
\mathbf{x}} Z(\g^x)$ is also a two-dimensional vector space, and so
the map $\gamma$ in \eqref{eq:center-surjection} is injective (since
it is surjective).  The result then follows from
Theorem~\ref{thm:finite-tilde-X}.
\end{proof}

In the special case $\g=\lsl_2$, $k = \C$, the irreducible
finite-dimensional representations of $\On(\lsl_2)$ were described
in \cite{DR} as follows. Let $\{e,h,f\}$ be an $\lsl_2$-triple and
define $X = (e+f)\otimes 1$ and $Y = e\otimes t + f \otimes t^{-1}
$. Then $\On(\lsl_2)$ is generated by $X,Y$. Furthermore, if $V$ is
an irreducible finite-dimensional representation of $\On(\lsl_2)$,
then $X$ and $Y$ are diagonalizable on $V$, and there exists an
integer $d \geq 0$ and scalars $\gamma, \gamma^* \in k$ such that
the set of distinct eigenvalues of X (resp. Y ) on $V$ is $\left\{d
- 2i + \gamma : 0 \leq i \leq d\right\}$ (resp. $\left\{d - 2i +
\gamma^* : 0 \leq i \leq d\right\}$) {\cite[Corollary~2.7]{Ha}}. The
ordered pair $(\gamma, \gamma^*)$ is called the \textit{type} of
$V$. Replacing $X, Y$ by $X - \gamma I, Y - \gamma^*I$ (in the
universal enveloping algebra $U(\On(\lsl_2))$ of $\On(\lsl_2)$) the
type becomes $(0, 0)$.

Let $\ev_{x_1}V_1, \ldots , \ev_{x_n}V_n$ denote a finite sequence
of evaluation modules for $\On(\lsl_2)$, and $V$ the evaluation
module $\ev_{x_1} V_1 \otimes \cdots  \otimes \ev_{x_n} V_n$. Any
module
 that can be obtained from $V$
by permuting the order of the factors and replacing any number of
the $x_i$'s with their multiplicative inverses will be called
\emph{equivalent} to $V$. The classification of irreducible
finite-dimensional $\On(\lsl_2)$-modules of type $(0, 0)$ is
described in \cite{DR} as follows.

\begin{prop}\mbox{}
\begin{enumerate}
\item {\cite[Theorem~6]{DR}} Every non-trivial irreducible finite-dimensional
$\On(\lsl_2)$-module of type $(0, 0)$ is isomorphic to a tensor
product of evaluation modules.
\item {\cite[Proposition~5]{DR}} Let $\ev_{x_1} V_1, \ldots,
\ev_{x_n} V_n$ denote a finite sequence of evaluation modules for
$\On(\lsl_2)$, and consider the $\On(\lsl_2)$-module $\ev_{x_1} V_1
\otimes \cdots  \otimes \ev_{x_n} V_n$. This module is irreducible
if and only if $x_1, x_1\inv,
 \ldots,  x_n, x_n\inv$ are pairwise distinct.
\item {\cite[Proposition~5]{DR}} Let $U$ and $V$ denote tensor
products of finitely many evaluation modules for $\On(\lsl_2)$.
Assume each of $U$, $V$ is irreducible as an $\On(\lsl_2)$-module.
Then the $\On(\lsl_2)$-modules $U$ and $V$ are isomorphic if and
only if they are equivalent.
\end{enumerate}
\end{prop}

These results are immediate consequences of
Proposition~\ref{prop:onsager-reps}.  The type $(0,0)$
representations are precisely the evaluation representations
$\ev_\Psi$ such that $\left\{ \pm 1 \right\} \cap \supp \Psi =
\emptyset$.  We note that under the previous definition of
evaluation representations appearing in the literature (see
Remark~\ref{rem:new-eval-rep}), \emph{only} type $(0,0)$
representations are evaluation representations. However, using
Definition~\ref{def:evaluation-rep}, \emph{all} irreducible
finite-dimensional representations (of arbitrary type) are
evaluation representations. The key is that we allow one-dimensional
representations of $\g^x$ for $x \in \{\pm 1\}$, where $\g^x$ is
one-dimensional. Thus our definition allows for a more uniform
description of the representations.  The condition that the points
$x_1, x_1^{-1}, \dots, x_n, x_n^{-1}$ be pairwise distinct also
follows automatically as in the case of multiloop algebras.

\subsection{A non-abelian example}
\label{subsec:applications-nonabelian}

Let $\g = \mathfrak{so}_8$, $X=\mathbb{P}^1\setminus \{0,1,\infty\}$
and $\Gamma=S_3$ as in Example~\ref{eg:nonabelian}, and let $\frM =
M(X,\g)^\Gamma$. Since $\frM$ is perfect, by
Corollary~\ref{cor:perfectreps} all irreducible finite-dimensional
representations of $\frM$ are evaluation representations and these
are naturally enumerated by $\mathcal{E}$. We identify $\Gamma=S_3$
with the permutations of the set $\{0,1,\infty\}$ and use the usual
cycle notation for permutations. For instance, $(0\, \infty)$
denotes the permutation given by $0 \mapsto \infty$, $\infty \mapsto
0$, $1 \mapsto 1$.  A straightforward computation shows that the
points with non-trivial stabilizer are listed in the table below
(note that $\{0,1,\infty\} \not \in X$). Each $\g^x$ is the fixed
point algebra of a diagram automorphism of $\g$ and thus a simple
Lie algebra of type $B_3$ or $G_2$.

\medskip
\begin{center}
\begin{tabular}{|c|c|c|} \hline
  $x$ & $\Gamma_x$ & Type of $\g^x$ \\ \hline
  $-1$ & $\{\Id, (0\, \infty)\} \cong \Z_2$ & $B_3$ \\
  2 & $\{\Id, (1\, \infty)\} \cong \Z_2$ & $B_3$ \\
  $\frac{1}{2}$ & $\{\Id, (0\, 1)\} \cong \Z_2$ &  $B_3$ \\
  $e^{\pm \pi i/3}$ & $\{\Id, (0\, 1\, \infty), (0\,
  \infty\, 1)\} \cong \Z_3$ & $G_2$ \\ \hline
\end{tabular}
\end{center}
\medskip

Furthermore, the sets $\{-1,2,\frac{1}{2}\}$ and $\{e^{\pi i/3},
e^{-\pi i/3}\}$ are $\Gamma$-orbits.  We thus see that the
representation theory of $\frM$ is quite rich.  Elements of
$\mathcal{E}$ can assign to the three-element orbit (the isomorphism
class of) any irreducible representation of the simple Lie algebra
of type $B_3$, to the two-element orbit any irreducible
representation of the simple Lie algebra of type $G_2$ and to any of
the other (six-element) orbits, any irreducible representation of
the simple Lie algebra $\g$ of type $D_4$.

\section*{Acknowledgements}
The authors thank Michael Lau for having explained the results of
a preprint of \cite{Lau} to them before it was posted, and Oliver
Schiffmann and Mark Haiman for helpful discussions concerning the
geometry of algebraic varieties.  They would also like to thank
Geordie Williamson, Ben Webster and Daniel Juteau for useful
conversations.


\providecommand{\bysame}{\leavevmode\hbox
to3em{\hrulefill}\thinspace}
\providecommand{\MR}{\relax\ifhmode\unskip\space\fi MR }
\providecommand{\MRhref}[2]{%
  \href{http://www.ams.org/mathscinet-getitem?mr=#1}{#2}
} \providecommand{\href}[2]{#2}

\end{document}